\documentclass[11pt]{article}
\usepackage[tbtags]{amsmath}
\usepackage{amssymb}
\usepackage{amsthm}
\usepackage[misc]{ifsym}
\usepackage{cases}
\usepackage{xcolor}
\usepackage{mathtools}
\usepackage{mathrsfs}
\usepackage{graphicx}
\usepackage{subcaption} 

\numberwithin{equation}{section}
\normalsize
\title{\bf General Linear-Quadratic Mean Field Stochastic Differential Game with Common Noise: a Direct Method \thanks{This work is supported by National Key R\&D Program of China (2022YFA1006104), National Natural Science Foundations of China (12471419, 12271304), and Shandong Provincial Natural Science Foundations (ZR2024ZD35, ZR2022JQ01).}}
\author{\normalsize  Yu Si\thanks{\it School of Mathematics, Shandong University, Jinan 250100, P.R. China, E-mail: 202112003@mail.sdu.edu.cn} , Jingtao Shi\thanks{\it Corresponding author. School of Mathematics, Shandong University, Jinan 250100, P.R. China, E-mail: shijingtao@sdu.edu.cn}}
\date{}
\newtheorem{myprob}{Problem}[section]
\newtheorem{mypro}{Proposition}[section]
\newtheorem{mythm}{Theorem}[section]

\newtheorem{mylem}{Lemma}[section]
\newtheorem{myassump}{Assumption}[section]
\newtheorem{myremark}{Remark}[section]

\begin{document}

\maketitle

\noindent{\bf Abstract:}\quad This paper investigates a class of general linear-quadratic mean field games with common noise, where the diffusion terms of the system contain the state variables, control variables, and the average state terms. We solve the problem using both the direct method and the fixed-point method in the paper. First, by using the variational method to solve a finite $N$-players game problem, we obtain the necessary and sufficient conditions for the centralized open-loop Nash equilibrium strategy. Subsequently, by employing the decoupling technique, we derive the state feedback representation of the centralized open-loop Nash equilibrium strategy in terms of Riccati equations. Next, by studying the asymptotic solvability of the Riccati equations, we construct the decentralized open-loop asymptotic Nash equilibrium strategies. Finally, through some estimates, we prove the asymptotic optimality of the decentralized open-loop Nash equilibrium strategies. Moreover, we find that the decentralized open-loop asymptotic Nash equilibrium strategies obtained via the fixed-point method are identical to those derived using the direct method. Finally, we apply the main results of this paper to a concrete production planning example.

\vspace{2mm}


\vspace{2mm}


\section{Introduction}

In recent years, {\it mean field games} (MFGs) have garnered significant attention in mathematics and control engineering, with broad applications spanning engineering, finance, and social sciences. A defining feature of MFGs is that in large-population systems, the influence of any individual participant becomes negligible, while the collective behavior of the population remains non-negligible. Due to the vast number of participants, their mutual perception and interactions exhibit high complexity. Consequently, centralized strategies relying on complete information about all participants are typically infeasible for any given agent. However, MFGs offer a solution framework: through mean field approximation, decentralized strategies can be designed using only local state information. Although such strategies do not guarantee precise optimality, they yield asymptotic-optimal performance when the population size is sufficiently large.

The past research mainly has two approaches. One approach starts by formally solving an $N$-agents game to obtain a large coupled solution equation system. The subsequent step involves designing decentralized optimal strategy by taking $N \rightarrow \infty$ (Lasry and Lions \cite{Lasry-Lions-2007}), which can be called the direct (bottom-up) approach. The interested readers can also refer to \cite{Cong-Shi-2024,  Huang-Zhou-2020, Si-Shi-2025-OCAM, Wang-2024, Wang-Zhang-Zhang-2022}. Another route is to solve an optimal control problem of a single agent by replacing the state average term with a limiting process and formalize a fixed point problem to determine the limiting process, and this is called the fixed point (top-down) approach. This kind of method is also called {\it Nash certainty equivalence} (NCE) (Huang et al. \cite{Huang-Caines-Malhame-2007, Huang-Malhame-Caines-2006}).
The interested readers refer to \cite{Bensoussan-Feng-Huang-2021, Hu-Huang-Li-2018,Huang-Wang-2016, Li-Nie-Wu-2023}. Further analysis of MFGs and related topics can refer to \cite{Bensoussan-Frehse-Yam-2013, Carmona-Delarue-2013, Huang-2010,  Moon-Basar-2018}.

Mean field games with common noise represent a significant category of game problems where participants are interconnected and share a common stochastic noise (Carmona et al. \cite{Carmona-Delarue-Lacker-2016}). Since common noise typically represents the influence of external factors that affect all participants, there are many real-life examples, especially in finance and economics. For instance, multi-agent systems may be subject to disturbances from the external environment, or all investors in a financial market may be affected by government policies (\cite{Carmona-Fouque-Sun-2015}, \cite{Souganidis-Zariphopoulou-2024}). Lacker and Flem \cite{Lacker-Flem-2023} provides a comprehensive closed-loop convergence framework for mean field games with common noise by introducing the concept of weak mean field equilibrium. Huang et al. \cite{Huang-Wang-Wang-Wang-2023} employed a novel decoupling technique to investigate the control problem of {\it linear-quadratic} (LQ) stochastic backward large population system with partial information and common noise. Hua and Luo \cite{Hua-Luo-2024} explores a class of LQ extended mean field games with common noises, where the mean field interactions occur solely through the (conditional) expectations of states and controls, and allows for nonlinear variations in the state coefficients and cost functionals. Si and Shi \cite{Si-Shi-2025} employed a fixed-point approach to analyze a class of general linear-quadratic mean-field Stackelberg games with common noise.

Recently, Li et al. \cite{Li-Nie-Wang-Yan-2024} uses the fixed-point method to study a class of general LQ mean field games with partial information and common noise, while Wang et al. \cite{Wang-Zhang-Zhang-2022} employs the direct method to analyze a class of general LQ mean field social optimum problems. In this paper, we apply the direct method to consider a class of general LQ mean field games with common noise. Compared with the existing literatures, the contributions of this paper are listed as follows.
\begin{itemize}
  \item We introduced a more general LQ mean field games with common noise. The weight matrix of the control in the cost functional is indefinite. And the common noise can be observed by all agents, and the diffusion terms of the system contain the state variables, control variables, and the average state terms. Due to the generality of the diffusion term, there are some difficulties in solving the optimal open-loop feedback representation strategies for the $N$-agents game. Therefore, we present a new decoupling form (\ref{diffusion coefficient of p_j}), which is different from that in Liang et al. \cite{Liang-Wang-Zhang-2025}.
  \item  We use the direct method to solve the problem. First, we use the variational method to solve a finite $N$-agents game problem. Then, by studying the asymptotic solvability of Riccati and Lyapunov equations, we derived a feedback representation of the asymptotic open-loop Nash equilibrium strategy. We also presented two conditions to ensure the well-posedness of Riccati equation (\ref{K equation}).
  \item By comparing with the fixed-point approach, we find that the decentralized open-loop asymptotic Nash equilibrium strategies obtained via the fixed-point method are identical to those derived using the direct method.
\end{itemize}

The rest of this paper is organized as follows. In Section 2, we formulate our problem. In Section 3, we first study the $N$-agents game problem, present the optimal centralized open-loop Nash equilibrium strategy, then construct the asymptotically optimal decentralized open-loop Nash equilibrium strategy, and compare it with results from the fixed-point method. Then, In Section 4, we apply the main results of this paper to a concrete production planning example. Finally, some concluding remarks are given in Section 5.

\section{Problem formulation}

Firstly, we introduce some notations that will be used throughout the paper. We consider a finite time interval $[0, T]$ for a fixed $T > 0$.
Let $\big(\Omega, \mathcal{F}, \left\{\mathcal{F}_t\right\}_{t\geq0},\mathbb{P}\big)$ be a complete filtered probability space, on which
$\left\{W_0(s), W_i(s),0 \leq s \leq t, 1 \leq i \leq N\right\}$ is a standard $(N+1)$-dimensional Brownian motion. 
Let $\mathbb{R}^n$ be an $n$-dimensional Euclidean space with norm and inner product being defined as $|\cdot|$ and $\langle\cdot, \cdot\rangle$, respectively. 
Next, we introduce three necessary spaces frequently used in this paper. A bounded, measurable function $f(\cdot):[0, T] \rightarrow \mathbb{R}^n$ is denoted as $f(\cdot) \in L^{\infty}(0, T; \mathbb{R}^n)$. $C\left([0, T], \mathbb{R}^n\right)$ is the space of all $\mathbb{R}^n$-valued continuous functions defined on $[0, T ]$.
An $\mathbb{R}^n$-valued, $\mathbb{F}$-adapted stochastic process $f(\cdot): \Omega \times [0, T] \rightarrow \mathbb{R}^n$ satisfying $\mathbb{E} \int_0^T |f(t)|^2 dt < \infty$ is denoted as $f(\cdot) \in L_{\mathbb{F}}^2(0, T; \mathbb{R}^n)$. Similarly, an $\mathbb{R}^n$-valued, $\mathcal{F}_{T}$-measurable random variable $\xi$ with $\mathbb{E} [|\xi|^2] < \infty$ is denoted as $\xi \in L_{\mathcal{F}_T}^2(\Omega, \mathbb{R}^n)$.
For any random variable or stochastic process $X$ and filtration $\mathcal{H}$, $\mathbb{E}X$ and $\mathbb{E}[X|\mathcal{H}]$ represent the mathematical expectation and conditional mathematical expectation of $X$, respectively. For a given vector or matrix \(M\), let \(M^\top\) represent its transpose. We denote the set of symmetric \(n \times n\) matrices (resp. positive semi-definite matrices) with real elements by \(\mathcal{S}^n\) (resp. \(\mathcal{S}_+^n\)), and $\|\cdot\|_Q^2 :=\langle Q \cdot, \cdot\rangle$ for any $Q \in \mathcal{S}^n$. If \(M \in \mathcal{S}^n\) is positive (semi) definite, we abbreviate it as \(M > (\geq) 0\). For a positive constant \(k\), if \(M \in \mathcal{S}^n\) and \(M > kI\), we label it as \(M \gg 0\).

Now, let us focus on a large population system comprised of $N$ individual agents, denoted as $\left\{\mathcal{A}_i\right\}_{1 \leq i \leq N}$. The state process $x_i(\cdot)\in\mathbb{R}^n$ of the agent $\mathcal{A}_i$ is given by the following linear {\it stochastic differential equation} (SDE):
\begin{equation}\label{state equation}
\left\{\begin{aligned}
	d x_i(t)=&\left(A(t) x_i(t)+B(t) u_i(t)+E(t) x^{(N)}(t)+f(t)\right) d t\\
	&+  \left(C(t) x_i(t)+D(t) u_i(t)+F(t) x^{(N)}(t)+g(t)\right) d W_i \\
	&+\left(\widetilde{C}(t) x_i(t)+\widetilde{D}(t) u_i(t)
	+\widetilde{F}(t)x^{(N)}(t)+\widetilde{g}(t)\right) d W_0, \\
	x_i(0)=&\ \xi_i,
\end{aligned}\right.
\end{equation}
where $\xi_i \in \mathbb{R}^n$ represent the initial states, $u_i(\cdot)\in\mathbb{R}^k$ denote the control process of the agent $\mathcal{A}_i$ and $x^{(N)}(\cdot) := \frac{1}{N} \sum_{i=1}^N x_i(\cdot)$ signifies the average state of the agents. In addition, $W_0$ represent the common noise of system and the information generated by the common noise is defined by
$\mathcal{F}_t^{0}:=\sigma\left\{W_0(s), 0 \leq s \leq t\right\}$. The coefficients $A(\cdot)$, $B(\cdot)$, $E(\cdot)$, $C(\cdot)$, $D(\cdot)$, $F(\cdot)$, $\widetilde{C}(\cdot)$, $\widetilde{D}(\cdot)$, $\widetilde{F}(\cdot)$ are deterministic matrix-valued functions with compatible dimensions and $f(\cdot)$, $g(\cdot)$, $\widetilde{g}(\cdot)$ are $\mathcal{F}_t^{0}$-adapted processes in $\mathbb{R}^n$.

For the agent $\mathcal{A}_i$, the cost functional is defined by
\begin{equation}\label{cost functional}
	\begin{aligned}
        \mathcal{J}_i\left(u_i(\cdot), u_{-i}(\cdot)\right)=&\ \frac{1}{2}\mathbb{E}\left[ \int_0^T\left(\left\|x_i(t)-\Gamma_1(t) x^{(N)}(t)-\eta_1(t)\right\|_Q^2+\left\|u_i(t)-\eta_2(t)\right\|_R^2\right) d t\right.\\
        &\qquad \left.+\left\|x_i(T)-\Gamma_0 x^{(N)}(T)-\eta_{0}\right\|_{G}^2\right],
    \end{aligned}
\end{equation}
where $Q(\cdot)$, $R(\cdot)$ and $\Gamma_1(\cdot)$ are bounded deterministic matrix-valued functions with compatible dimensions. And we define the admissible centralized strategy of the agent $\mathcal{A}_i$ by
$$
\mathcal{U}_{i}^c=\left\{u_i(\cdot) \mid u_i(\cdot) \in L_{\mathcal{F}_t}^2\big(0, T ; \mathbb{R}^k\big)\right\},
$$
where $\mathcal{F}_t:=\sigma\left\{\xi_i, W_0(s), W_i(s),0 \leq s \leq t, 1 \leq i \leq N\right\}$.
Moreover, we introduce the following assumptions of coefficients.

\begin{myassump}\label{A1}
	The coefficients satisfy the following conditions:
	
	(i) $A(\cdot),C(\cdot), \widetilde{C}(\cdot),E(\cdot), F(\cdot), \widetilde{F}(\cdot) \in L^{\infty}\left(0, T ; \mathbb{R}^{n \times n}\right)$, $f(\cdot),g(\cdot),\widetilde{g}(\cdot)\in L^{2}_{\mathcal{F}^0_t}\left(0, T ; \mathbb{R}^{n}\right)$ and $B(\cdot), D(\cdot), \widetilde{D} \in L^\infty\left(0, T; \mathbb{R}^{n \times k}\right) $;
	
	(ii) $Q(\cdot) \in L^{\infty}\left(0, T ; \mathcal{S}^n\right)$, $R(\cdot) \in L^{\infty}\left(0, T ; \mathcal{S}^k\right)$, $\Gamma_1(\cdot) \in L^{\infty}\left(0, T ; \mathbb{R}^{n\times n}\right)$, $\eta_1(\cdot)\in L^{2}_{\mathcal{F}^0_t}\left(0, T ; \mathbb{R}^n\right)$, $\eta_2(\cdot)\in L^{2}_{\mathcal{F}^0_t}\left(0, T ; \mathbb{R}^k\right)$, $G \in \mathcal{S}^n$,  $\Gamma_0 \in \mathbb{R}^{n\times n}$, $\eta_0 \in L^{2}_{\mathcal{F}^0_T}\left(\Omega, \mathbb{R}^n\right)$ and $Q(\cdot) \geq 0$, $G \geq 0$.
	
	(iii) $\xi_i$, $i=1,\cdots,N$ are mutually independent and have
	the same distribution. And the noises $\sigma\{W_0(s), W_i(s), 0\leq t \leq T,  i=1,\cdots, N\}$ and the initial states $\sigma\{\xi_i, i=1,\cdots,N\}$ are independent of each other.
\end{myassump}

We mention that under Assumption \ref{A1}, the system (\ref{state equation}) admits a unique solution. And then, the cost functionals (\ref{cost functional}) are well-defined. Then, we need to solve this centralized problem.

\begin{myprob}\label{problem centralized}
	Finding centralized open-loop Nash equilibrium $\left(u_1^*(\cdot), \cdots, u_N^*(\cdot)\right)$, $u_i^*(\cdot)\in\mathcal{U}^c_i, i=1,\cdots, N$ for the system (\ref{state equation})-(\ref{cost functional}), i.e., for $1 \leq i \leq N$, we have
	\begin{equation*}
		\mathcal{J}_i\left(u^*_i(\cdot), u^*_{-i}(\cdot)\right)=\inf _{u_i(\cdot) \in\, \mathcal{U}_i^{c}} \mathcal{J}_i\left(u_i(\cdot), u^*_{-i}(\cdot)\right).
	\end{equation*}
\end{myprob}

Since it is difficult to obtain complete information in real life, it is challenging to derive an accurate centralized strategy. Therefore, mean field game theory focuses on designing decentralized strategies that rely solely on agents' local information, which ensuring these strategies achieve asymptotically optimal performance. We first define the admissible decentralized strategy of the agent $\mathcal{A}_i$ by
$$
\mathcal{U}_{d}^i=\left\{u_i(\cdot) \mid u_i(\cdot) \in L_{\mathcal{F}^i_t}^2\big(0, T ; \mathbb{R}^k\big)\right\},
$$
where $\mathcal{F}_t^i:=\sigma\left\{\xi_i, W_i(s),W_0(s), 0 \leq s \leq t\right\} , 1 \leq i \leq N$.
As the second step of the direct approach, we will design the decentralized asymptotic Nash equilibrium based on the the solutions of Problem \ref{problem centralized}. The problem is stated below.

\begin{myprob}\label{problem decentralized}
	Finding decentralized open-loop $\varepsilon$-Nash equilibrium $\left(u_1^*(\cdot), \cdots, u_N^*(\cdot)\right)$, $u_i^*(\cdot)\in\mathcal{U}^d_i, i=1,\cdots, N$ for the system (\ref{state equation})-(\ref{cost functional}), i.e., for $1 \leq i \leq N$, we have
	\begin{equation*}
		\mathcal{J}_i\left(u^*_i(\cdot), u^*_{-i}(\cdot\right)\leq\inf _{u_i(\cdot) \in\, \mathcal{U}_i^{c}} \mathcal{J}_i\left(u_i(\cdot), u^*_{-i}(\cdot)\right)+\varepsilon,
	\end{equation*}
	where $\varepsilon=O(\frac{1}{\sqrt{N}})$.
\end{myprob}

\section{Main Result}

\subsection{The $N$-player Nash game and open-loop centralized Nash equilibrium}

Note that the control weight $R(\cdot)$ in the cost functional is indefinite. To ensure the well-posedness of Problem 2.1, convexity properties need to be given. Based on Proposition 2.2.3 in Sun and Yong \cite{Sun-Yong-2020}, we first present equivalent conditions for the convexity of cost functional $\mathcal{J}_i\left(u_i(\cdot), u_{-i}(\cdot\right)$.
\begin{mylem}
    In Problem \ref{problem centralized}, for $i=1, \cdots, N$, $\mathcal{J}_i\left(u_i(\cdot), u_{-i}(\cdot\right)$ is convex in $u_i$ if and only if for any $u_i(\cdot)\in \mathcal{U}^c_i$,
    $$
    \begin{aligned}
        &\mathbb{E}\left[ \int_0^\top  \left(\left\langle Q\left(y_i-\Gamma_1 y^{(N)}\right), y_i-\Gamma_i y^{(N)}\right\rangle+\left\langle R u_i, u_i\right\rangle\right) d t\right. \\
        &\qquad +\left\langle G\left(y_i(T)-\Gamma_0 y^{(N)}(T)\right), y_i(T)-\Gamma_0 y^{(N)}(T)\right\rangle\bigg] \geq 0,
    \end{aligned}
    $$
    where $y_i(\cdot)$ satisfies the following equation
    $$
    \left\{\begin{aligned}
    	dy_i=&\left(Ay_i+B u_i+Ey^{(N)}\right) d t+  \left(Cy_i+D u_i+F y^{(N)}\right) d W_i \\
    	&+\left(\widetilde{C}y_i+\widetilde{D} u_i+\widetilde{F} y^{(N)}\right) d W_0, \\
    	y_i(0)=&\ 0,
    \end{aligned}\right.
    $$
    $$
    \left\{\begin{aligned}
    	dy_j=&\left(Ay_j+Ey^{(N)}\right) d t+  \left(Cy_j+F y^{(N)}\right) d W_j+\left(\widetilde{C}y_j+\widetilde{F} y^{(N)}\right) d W_0, \\
    	y_j(0)=&\ 0,
    \end{aligned}\right.
    $$
    with $y^{(N)}:=\frac{1}{N}\sum_{k=1}^{N}y^i$.
\end{mylem}
By using the variational analysis, we obtain the necessary and sufficient condition of the existence of centralized Nash equilibria for Problem \ref{problem centralized}.

\begin{mythm}\label{theorem open loop centralized}
	Let Assumption \ref{A1} hold, Problem \ref{problem centralized} has a centralized Nash equilibrium $(u^*_1(\cdot),\\\cdots, u^*_N(\cdot))$ if and only if for $i=1, \cdots, N$, $\mathcal{J}_i\left(u_i(\cdot), u_{-i}(\cdot)\right)$ is convex in $u_i$ and the following {\it forward-backward SDEs} (FBSDEs)
	\begin{equation}\label{Hamiltonian system centralized}
		\left\{\begin{aligned}
			d x^*_i = &\left(A x^*_i+B u^*_i+E x^{*(N)}+f\right) d t+  \left(C x^*_i+D u^*_i+F x^{*(N)}+g\right) d W_i \\
			&+\left(\widetilde{C}x^*_i+\widetilde{D} u^*_i +\widetilde{F}x^{*(N)}+\widetilde{g}\right) d W_0, \\
			d p_i^i = &-\bigg( A^\top p_i^i + C^\top q_{ii}^i + \widetilde{C}^\top q_{i0}^i + E^\top p^{i(N)} + F^\top q^{i(N)} + \widetilde{F}^\top q_0^{i(N)} \\
			& +\left(I_n-\frac{\Gamma_1}{N}\right)^\top Q \left( x_i^* - \Gamma_1 x^{*(N)} - \eta_1 \right) \bigg) dt + \sum_{k=1}^N q_{ik}^i \, dW_k + q_{i0}^i \, dW_0,\\
			d p_j^i = &-\bigg(A^\top p_j^i+C^\top q_{j j}^i+\widetilde{C}^\top q_{j 0}^i+E^\top p^{i{(N)}}+F^\top q^{i{(N)}}+\widetilde{F}^\top q_0^{i(N)} \\
			& -\frac{\Gamma_1^\top}{N}Q \left( x_i^* - \Gamma_1 x^{*(N)} - \eta_1 \right) \bigg) dt + \sum_{k=1}^N q_{jk}^i \, dW_k + q_{j0}^i \, dW_0, \\
			\bar{x}_i^*(0)&=\xi_i,\\
			p_i^i(T)&=\left(I_n-\frac{\Gamma_0}{N}\right)^\top G\left(x^*_i(T)-\Gamma_0 x^{*(N)}(T)-\eta_0\right),\\
			p_j^i(T)&=-\Gamma_0^\top \frac{G}{N}\left(x^*_i(T)-\Gamma_0 x^{*(N)}(T)-\eta_0\right),
		\end{aligned}\right.
	\end{equation}
	admit adapted solution $(x_i^*(\cdot), p_j^i(\cdot), q_{jk}^i(\cdot), q_{j0}^i(\cdot))\in L^2_{\mathcal{F}_t}(0,T;\mathbb{R}^n)\times L_{\mathcal{F}_t}^2\left(0, T ; \mathbb{R}^n\right) \times L_{\mathcal{F}_t}^2\left(0, T ; \mathbb{R}^{n}\right)\times L_{\mathcal{F}_t}^2\left(0, T ; \mathbb{R}^{n}\right)$, where $p^{i(N)}:=\frac{1}{N}\sum_{k=1}^{N}p^i_k$, $q^{i(N)}:=\frac{1}{N}\sum_{k=1}^{N}q^i_{kk}$, $q_0^{i(N)}:=\frac{1}{N}\sum_{k=1}^{N}q^i_{k0}$, satisfying the stationary condition:
	\begin{equation}\label{optimal open loop condition of centralized}
		B^\top p_i^i+D^\top q_{i i}^i+\widetilde{D}^\top q_{i 0}^i+R u_i^*-R \eta_2=0.
	\end{equation}
\end{mythm}

\begin{proof}
	Suppose $(u^*_1(\cdot),\cdots, u^*_N(\cdot))$ is a centralized Nash equilibrium of Problem \ref{problem centralized} and $(x^*_1(\cdot),\cdots, \\x^*_N(\cdot))$ is the corresponding optimal trajectory. Given $i$, for any $u_i(\cdot) \in \mathcal{U}^c_i$ and $\forall$ $\varepsilon>0$, we denote
	$$
	u_i(\cdot)^\varepsilon:=u^*_i(\cdot)+\varepsilon v_i(\cdot) \in \mathcal{U}^c_i,
	$$
	where $v_i(\cdot)=u_i(\cdot)-u^*_i(\cdot)$.
	
	Let $(\bar{x}^\varepsilon_i(\cdot), \bar{x}^\varepsilon_j(\cdot))$, $j\neq i$ be the solution to the following perturbed state equations
	$$
	\left\{\begin{aligned}
		d x_i^{\varepsilon}=&\left(A x_i^{\varepsilon}+B u_i^{\varepsilon}+E x^{{\varepsilon}{(N)}}+f\right) d t+  \left(C x_i^{\varepsilon}+D u_i^{\varepsilon}+F x^{{\varepsilon}{(N)}}+g\right) d W_i
		\\&+\left(\widetilde{C} x_i^{\varepsilon}+\widetilde{D} u_i^{\varepsilon}+\widetilde{F} x^{{\varepsilon}{(N)}}+\widetilde{g}\right) d W_0 ,\\
		x_i^{\varepsilon}(0)=&\ \xi_i,
	\end{aligned}\right.
	$$
	$$
	\left\{\begin{aligned}
		d x_j^{\varepsilon}=&\left(A x_j^{\varepsilon}+B u_j^*+E x^{{\varepsilon}{(N)}}+f\right) d t+  \left(C x_j^{\varepsilon}+D u_j^*+F x^{{\varepsilon}{(N)}}+g\right) d W_i
		\\&+\left(\widetilde{C} x_j^{\varepsilon}+\widetilde{D} u_j^*+\widetilde{F} x^{{\varepsilon}{(N)}}+\widetilde{g}\right) d W_0, \\
		x_j^{\varepsilon}(0)=&\ \xi_j.
	\end{aligned}\right.
	$$
	Let $\delta x_i(\cdot):=\frac{x_i^{\varepsilon}(\cdot)-x_i^*(\cdot)}{\varepsilon},\delta x_j(\cdot):=\frac{x_j^{\varepsilon}(\cdot)-x_j^*(\cdot)}{\varepsilon}$, $j\neq i$. It can be verified that $\delta \bar{x}_i(\cdot)$ satisfies
	$$
	\left\{\begin{aligned}
		d\delta x_i=&\left(A\delta x_i+B v_i+E\delta x^{(N)}\right) d t+  \left(C\delta x_i+D v_i+F \delta x^{(N)}\right) d W_i \\
		&+\left(\widetilde{C}\delta x_i+\widetilde{D} v_i+\widetilde{F} \delta x^{(N)}\right) d W_0, \\
		\delta x_i(0)=&\ 0,
	\end{aligned}\right.
	$$
	and $\delta x_j(\cdot)$, $j\neq i$ satisfies
	$$
	\left\{\begin{aligned}
	    d\delta x_j=&\left(A\delta x_j+E\delta x^{(N)}\right) d t+  \left(C\delta x_j+F \delta x^{(N)}\right) d W_j+\left(\widetilde{C}\delta x_j+\widetilde{F} \delta x^{(N)}\right) d W_0 ,\\
	    \delta x_j(0)=&\ 0,
    \end{aligned}\right.
	$$
	where $\delta x^{(N)}:=\frac{1}{N}\sum_{k=1}^{N}\delta x^i_k$.
	Applying It\^o's formula to $\left\langle \delta x_i(\cdot), p_i(\cdot)\right\rangle+\sum_{j\neq i}\left\langle \delta x_j(\cdot), p_j(\cdot)\right\rangle$, we derive
	\begin{equation}\label{Ito of open loop}
		\begin{aligned}
			& \left\langle G \delta x_i-G \Gamma_0 \delta x^{(N)}, x^*_i-\Gamma_0 x^{*(N)}-\eta_0\right\rangle\\
			=&\ \mathbb{E}\left[ \int_0^T\left(\left\langle v_i, B^\top p_i^i+D^\top q_{i i}^i+\widetilde{D}^\top q_{i 0}^i\right\rangle-\left\langle\delta x_i, \left(I_n-\frac{\Gamma_1}{N}\right)^\top Q \left( x_i^* - \Gamma_1 x^{*(N)} - \eta_1 \right)\right\rangle \right.\right.\\
			&\qquad -\sum_{ j \neq i}\left\langle\delta x_j, \frac{\Gamma_1^\top}{N}Q  \left( x_i^* - \Gamma_1 x^{*(N)} - \eta_1 \right)\right\rangle\Bigg)dt\Bigg].
		\end{aligned}
	\end{equation}
	Then
	$$
	\begin{aligned}
		&\mathcal{J}_i\left(u_i^{\varepsilon}(\cdot) ; u_{-i}^*(\cdot)\right)-\mathcal{J}_i\left(u_i^*(\cdot) ; u_{-i}^*(\cdot)\right)\\
		=&\ \frac{\varepsilon}{2} \mathbb{E}\left[ \int_0^\top  \left(\left\langle Q\left(\delta x_i-\Gamma_1 \delta x^{(N)}\right), \delta x_i-\Gamma_i \delta x^{(N)}\right\rangle+\left\langle R v_i, v_i\right\rangle\right) d t\right. \\
		&\qquad +\left\langle G\left(\delta x_i(T)-\Gamma_0 \delta x^{(N)}\right), \delta x_i(T)-\Gamma_0 \delta x^{(N)}(T)\right\rangle\bigg]\\
		&+\varepsilon \mathbb{E}\left[ \int_0^T \left(\left\langle Q\left(\delta x_i-\Gamma_1 \delta x^{(N)}\right), x_i^*-\Gamma_1 x^{*(N)}-\eta_1\right\rangle+\left\langle R v_i, u_i^*-\eta_2\right\rangle\right) d t \right.\\
		&\qquad +\left\langle G\left(\delta x_i(T)-\Gamma_0 \delta x^{(N)}\right), x_i^*(T)-\Gamma_0 x^{*(N)}(T)-\eta_0\right\rangle\bigg]\\
		\equiv&\ \frac{\varepsilon}{2}X_1+\varepsilon X_2.
	\end{aligned}
	$$
	Due to the optimality of $u^*_i(\cdot)$ and the arbitrariness of $\varepsilon$, we have that $\mathcal{J}_i\left(u_i^{\varepsilon}(\cdot) ; u_{-i}^*(\cdot)\right)-\mathcal{J}_i\left(u_i^*(\cdot) ; u_{-i}^*(\cdot)\right) \geq 0$ if and only if $X_1 \geq 0$ and $X_2 = 0$. Noticing $X_1 \geq 0$ is equivalent to the convexity of $\mathcal{J}_i\left(u_i(\cdot), u_{-i}(\cdot\right)$ with respect to $u_i$. Then, simplifying the equation of $X_2$ with (\ref{Ito of open loop}), we have
	$$
	X_2=\mathbb{E} \int_0^T\left\langle B^\top p_i^i+D^T q_{i i}^i+\widetilde{D}^T q_{i 0}^i+R u_i^*-R\eta_2, v_i\right\rangle d t\,.
	$$
	Due to the arbitrariness of $v_i(\cdot)$, we obtain the condition (\ref{optimal open loop condition of centralized}). The proof is complete.
\end{proof}

For further analysis, we make the following assumption.
\begin{myassump}\label{A2}
	For $i=1, \cdots, N$, $\mathcal{J}_i\left(u_i(\cdot), u_{-i}(\cdot\right)$ is uniformly convex in $u_i$.
\end{myassump}
$\mathscr{f}$
Next, we employ a decoupling technique to analyze the solvability of FBSDEs (\ref{Hamiltonian system centralized}) and derive the state feedback form of the open-loop centralized Nash equilibrium (\ref{optimal open loop condition of centralized}).

Noting the terminal condition and structure of (\ref{Hamiltonian system centralized}), for each $i=1, \cdots, N$, we suppose
\begin{equation}\label{follower decouple form 1}
	    p_i^i(\cdot)=P_N(\cdot) x_i^*(\cdot)+K_N(\cdot) x^{*(N)}(\cdot)+\varphi_N(\cdot),
\end{equation}
\begin{equation}\label{follower decouple form 2}
	p_j^i(\cdot)=\Pi_N(\cdot) x_i^*(\cdot)+S_N(\cdot) x_j^*(\cdot)+M_N(\cdot) x^{*(N)}(\cdot)+\psi_N(\cdot),
\end{equation}
where $P_N(\cdot),K_N(\cdot),\Pi_N(\cdot),S_N(\cdot), M_N(\cdot)$ are deterministic differentiable functions satisfying\\ $P_N(T)=\left(I_n-\frac{\Gamma_0}{N}\right)^\top G$, $K_N(T)=-\left(I_n-\frac{\Gamma_0}{N}\right)^\top G\Gamma_0$, $\Pi_N(T)=-\frac{\Gamma_0^\top}{N}G$, $S_N(T)=0$, $M_N(T)=\frac{\Gamma_0^\top}{N}G\Gamma_0$, and $(\varphi_N(\cdot), \psi_N(\cdot))$ is an $\mathcal{F}^0_t$-adapted processes pair satisfying
$$
	d \varphi_N=\alpha_N d t+\beta_N d W_0, \quad \varphi_N(T)=-\left(I_n-\frac{\Gamma_0}{N}\right)^\top G \eta_0,
$$
$$
	d \psi_N=\gamma_N d t+\zeta_N d W_0 ,\quad \psi_N(T)=\frac{\Gamma_0^\top}{N} G \eta_0.
$$

\begin{myremark}
	The decoupled form (\ref{follower decouple form 2}) is different from previous literature \cite{Liang-Wang-Zhang-2025, Wang-2025} due to the presence of $C(\cdot)$ and $\widetilde{C}(\cdot)$. If $C(\cdot)=0$ and $\widetilde{C}(\cdot)=0$, then we can infer that $S(\cdot)=0$, and the decoupled form degenerates to the cases in previous literature.
\end{myremark}

By applying It\^o's formula to (\ref{follower decouple form 1}) and comparing the diffusion term coefficients, we first get
\begin{equation}\label{diffusion coefficient of p_i}
	\left\{\begin{aligned}
		&P_N\left(C x^*_i+D u^*_i+F x^{*(N)}+g\right)+ \frac{K_N }{N}\left(C x^*_i+D u^*_i+F x^{*(N)}+g\right)=q_{i i}^i, \\
		&\frac{K_N}{N}\left(C x^*_k+D u^*_k+F x^{*(N)}+g\right)=q_{i k}^i, \\
		&P_N\left(\widetilde{C} x^*_i+\widetilde{D} u^*_i+\widetilde{F} x^{*(N)}+\widetilde{g}\right)+K_N\left[\left(\widetilde{C}+\widetilde{F}\right) x^{*(N)}+\widetilde{D} u^{*(N)}+\widetilde{g}\right]+\beta_N=q_{i 0}^i,
	\end{aligned}\right.
\end{equation}
where $u^{*(N)}:=\frac{1}{N}\sum_{k=1}^{N}u^*_{i}$.
To obtain the feedback representation of the centralized open-loop Nash equilibrium, we need the following assumption.
\begin{myassump}\label{A3}
	$$
	\begin{aligned}
		& R+D^\top\left(P_N+\frac{K_N}{N}\right) D+\widetilde{D}^\top P_N \widetilde{D}>0,\\
		& R+D^\top\left(P_N+\frac{K_N}{N}\right) D+\widetilde{D}^\top (P_N+K_N) \widetilde{D}>0.
	\end{aligned}
	$$
\end{myassump}
Substituting (\ref{diffusion coefficient of p_i}) into (\ref{optimal open loop condition of centralized}), we first get
$$
u^{*(N)}=-\left(\mathcal{R}_N+\widetilde{D}^\top K_N \widetilde{D}\right)^{-1}\left[\left(\widetilde{P}_N+\widetilde{K}_N\right) x^{*(N)}+ \Phi_N\right],
$$
where
$$
\begin{aligned}
	& \mathcal{R}_N:=R+D^\top \left(P_N+\frac{K_N}{N}\right) D+\widetilde{D}^\top P_N \widetilde{D} ,\\
	& \widetilde{P}_N:=B^\top P_N+D^\top \left(P_N+\frac{K_N}{N}\right) C+\widetilde{D}^\top P_N \widetilde{C} ,\\
	& \widetilde{K}_{N}:=B^\top K_N+D^\top \left(P_N+\frac{K}{N}\right) F+\widetilde{D}^\top P_N \widetilde{F}+\widetilde{D}^\top K_N\left(\widetilde{C}+\widetilde{F}\right),\\
	& \Phi_N:=B^\top \varphi_N+D^\top \left(P_N+\frac{K_N}{N}\right) g+\widetilde{D}^\top \left(P_N+K_N\right) \widetilde{g}+\widetilde{D}^\top \beta_N-R\eta_2 ,
\end{aligned}
$$
and then we have
$$
u_i^*=-\mathcal{R}_N^{-1} \widetilde{P}_N x_i^*+\left[\mathcal{R}_N^{-1} \widetilde{P}_N-\left(\mathcal{R}_N+\widetilde{D}^\top K_N \widetilde{D}\right)^{-1}\left(\widetilde{P}_N+\widetilde{K}_N\right)\right] x^{*(N)}-\left(\mathcal{R}_N+\widetilde{D} K_N \widetilde{D}\right)^{-1} \Phi_N.
$$
Similarly, by applying It\^o's formula to (\ref{follower decouple form 2}) and comparing the diffusion coefficients, we first get
\begin{equation}\label{diffusion coefficient of p_j}
	\left\{\begin{aligned}
			q^i_{j i} &=\Pi_N\left(C x^*_i+D u^*_i+F x^{*(N)}+g\right)+\frac{M_N}{N}\left(C x^*_i+D u^*_i+F x^{*(N)}+g\right), \\
			q^i_{j j} &=S_N\left(C x^*_j+D u^*_j+F x^{*(N)}+g\right)+\frac{M_N}{N}\left(C x^*_j+D u^*_j+F x^{*(N)}+g\right) , \\
			q^i_{j k} &=\frac{M_N}{N}\left(C x^*_k+D u^*_k+F x^{*(N)}+g\right),\quad k \neq i, j,  \\
			q^i_{j0}&=\Pi_N\left(\widetilde{C} x^*_i+\widetilde{D} u_i^*+\widetilde{F} x^{*(N)}+\widetilde{g}\right)+S_N\left(\widetilde{C} x^*_j+\widetilde{D} u^*_j+\widetilde{F} x^{*(N)}+\widetilde{g}\right) \\
			&\quad +M_{N}\left[\left(\widetilde{C}+\widetilde{F}\right) x^{*(N)}+\widetilde{D} u^{*(N)}+\widetilde{g}\right]+\zeta_N .
	\end{aligned}\right.
\end{equation}
Combining equations (\ref{follower decouple form 1}), (\ref{follower decouple form 2}), (\ref{diffusion coefficient of p_i}) and (\ref{diffusion coefficient of p_j}), we obtain the expressions for $p^{i(N)}(\cdot)$, $q^{i(N)}(\cdot)$, and $q_0^{i(N)}(\cdot)$ as follows:
$$
p^{i(N)}=\left(\frac{P_N}{N}+\frac{N-1}{N} \Pi_N-\frac{S_N}{N}\right) x_i^*+\left(\frac{K_N}{N}+S_N+\frac{N-1}{N} M_N\right) x^{*(N)}+\frac{\varphi_N}{N}+\frac{N-1}{N} \psi_N,
$$
$$
\begin{aligned}
	q^{i(N)}=&\ \frac{1}{N}\left\{\left(P_N+\frac{K_N}{N}-S_N-\frac{M_N}{N}\right)\left(C-D \mathcal{R}_N^{-1} \widetilde{P}_N\right) x_i^*+\bigg[\left(NS_N+M_N\right)C\right.\\
	&+\left(P_N+\frac{K_N}{N}+(N-1) S_N+\frac{N-1}{N} M_N\right)F+\left(P_N+\frac{K_N}{N}-S_N-\frac{M_N}{N}\right)D\mathcal{R}_N^{-1} \widetilde{P}_N\\
	& -\left(P_N+\frac{K_N}{N}+(N-1) S_N+\frac{N-1}{N} M_N\right)D\left(\mathcal{R}_N+\widetilde{D}^\top K_N \widetilde{D}\right)^{-1}\left(\widetilde{P_N}+\widetilde{K}_N\right)\bigg]x^{*(N)} \\
	&\left. +\left(P_N+\frac{K_N}{N}+(N-1) S_N+\frac{N-1}{N} M_N\right)\left(g-D\left(\mathcal{R}_N+\widetilde{D} K_N \widetilde{D}\right)^{-1} \Phi_N\right)\right\},
\end{aligned}
$$
$$
\begin{aligned}
	q_0^{i(N)}=&\ \frac{1}{N}\bigg\{ \big(P_N+(N-1) \Pi_N-S_N\big)\left(\widetilde{C}-\widetilde{D} \mathcal{R}_N^{-1} \widetilde{P}_N\right) x_i^*+\bigg[\big(K_N+(N-1) M_N+N S_N\big)\widetilde{C} \\
	&+\big(P_N+K_N+(N-1)\left(M_N+S_N-\Pi_N\right)\big)\widetilde{F}+\big(P_N+(N-1) \Pi_N-S_N\big)\widetilde{D}\mathcal{R}_N^{-1} \widetilde{P}_N \\
	&\left.-\big(P_N+K_N+(N-1)\left(M_N+S_N-\Pi_N\right)\big)\widetilde{D}\left(\mathcal{R}_N+\widetilde{D}^\top K_N \widetilde{D}\right)^{-1}\left(\widetilde{P}_N+\widetilde{K}_N\right)\right] x^{*(N) } \\
	&+\big(P_N+K_N+(N-1)\left(M_N+S_N-\Pi_N\right)\big)\left[\widetilde{g}-\widetilde{D}\left(\mathcal{R}_N+\widetilde{D} K_N \widetilde{D}\right)^{-1} \Phi_N\right]\\
	&+\beta_N+\frac{N-1}{N} \xi_N\bigg\}.
\end{aligned}
$$
Then, by applying It\^o's formula to (\ref{follower decouple form 1}) and comparing the drift coefficients, we obtain the equation satisfied by the coefficients of the term $x^*_i(\cdot)$:
\begin{equation}\label{P_N equation}
	\left\{\begin{aligned}
		&\dot{P}_N  +P_N A+A^\top P_N+C^\top\left(P_N+\frac{K_N}{N}\right) C+\widetilde{C}^\top P_N \widetilde{C}  -\widetilde{P}_N^\top \mathcal{R}_N^{-1} \widetilde{P}_N \\
		& +E^\top \left(\frac{P_N}{N}+\frac{N-1}{N} \Pi_N-\frac{S_N}{N}\right)+ \frac{F^\top}{N}\left(P_N+\frac{K_N}{N}-S_N-\frac{M_N}{N}\right)\left(C-D \mathcal{R}_N^{-1} \widetilde{P}_N\right) \\
		& + \frac{\widetilde{F}^\top}{N}\big(P_N+( N-1) \Pi_N-S_N\big)\left(\widetilde{C}-\widetilde{D} \mathcal{R}_N^{-1} \widetilde{P}_N\right)+\left(I_n-\frac{\Gamma_1}{N}\right)^\top Q=0, \\
		&P_N(T) =\left(I_n-\frac{\Gamma_0}{N}\right)^\top G ,
	\end{aligned}\right.
\end{equation}
the equation satisfied by the coefficients of the term $x^{*(N)}(\cdot)$:
\begin{equation}\label{K_N equation}
\left\{\begin{aligned}
	& \dot{K}_N+K_N(A+E)+P_N E+A^\top K_N+C\left(P_N+\frac{K_N}{N}\right) F+\widetilde{C}^\top K_N \widetilde{C}\\
	& +\widetilde{C}^\top \left(P_N+K_N\right) \widetilde{F}+\widetilde{P}_N^\top \mathcal{R}_N^{-1} \widetilde{P}_N-\left(\widetilde{P}_N^\top+K_N B +\widetilde{C}^\top K_N \widetilde{D} \right)\left(\mathcal{R}_N+\widetilde{D}^\top K_N \widetilde{D}\right)^{-1} \\
	& \times\left(\widetilde{P}_N+\widetilde{K}_N\right)+E^\top \left(\frac{K_n}{N}+S_N+\frac{N-1}{N} M_n\right)+F^\top \left(S_N+\frac{M_N}{N}\right)C\\
    & +\widetilde{F}^\top\left(\frac{K_n}{N}+S_N+\frac{N-1}{N} M_n\right)\widetilde{C} +F^\top \left(\frac{P_N}{N}+\frac{K_N}{N^2}+\frac{N-1}{N}S_{N}+\frac{N-1}{N^2}M_N\right) F \\
	& +\widetilde{F}^\top \left(\frac{P_N}{N}+\frac{K_N}{N}+\frac{N-1}{N}\left(\Pi_N+S_N+M_N\right)\right)\widetilde{F}^\top \\
	& +\left[F^\top \left(\frac{P_N}{N}+\frac{K_N}{N^2}-\frac{S_{N}}{N}-\frac{M_N}{N^2}\right) D+\widetilde{F}^\top \left(\frac{P_N}{N}+\frac{N-1}{N}\Pi_N-\frac{S_{N}}{N}\right)\widetilde{D}\right]\mathcal{R}_N^{-1} \widetilde{P}_N \\
	& -\left[F^\top \left(\frac{P_N}{N}+\frac{K_N}{N^2}+\frac{N-1}{N}S_{N}+\frac{N-1}{N^2}M_N\right)D\right.\\
    & \left.+\widetilde{F}^\top \left(\frac{P_N}{N}+\frac{K_N}{N}+\frac{N-1}{N}\left(\Pi_N+M_N+S_N\right)\right)\widetilde{D}\right]\left(\mathcal{R}_N+\widetilde{D}^\top K_N \widetilde{D}\right)^{-1}\\
	& \times\left(\widetilde{P}_N+\widetilde{K}_N\right)-\left(I_n-\frac{\Gamma_1}{N}\right)^\top Q \Gamma_1 =0,\\
	& K_N(T)=-\left(I_n-\frac{\Gamma_0}{N}\right)^\top G \Gamma_0,
\end{aligned}\right.
\end{equation}
and the equation satisfied by the non-homogeneous term:
\begin{equation}\label{varphi_N equation}
\left\{\begin{aligned}
	d \varphi_N= & -\left\{\left(A^\top+\frac{E^\top}{N}\right) \varphi_N+E^\top \frac{N-1}{N} \psi_N+\left(P_N+K_N\right) f-\bigg[\left(P_N+K_N\right) B\right.\\
	&+C^\top\left(P_N+\frac{K_N}{N}\right) D+\widetilde{C}^\top P_N \widetilde{D}+F^\top\left(\frac{P_N}{N}+\frac{K_N}{N^2}+\frac{N-1}{N} S_N+\frac{N-1}{N^2} M_N\right) D\\
	&+\widetilde{F}^\top\left(\frac{P_N}{N}+\frac{K_N}{N}+\frac{N-1}{N}\left(M_N+S_N+\Pi_N\right)\right) \widetilde{D}\bigg]\left(\mathcal{R}_N+\widetilde{D} K_N \widetilde{D}\right)^{-1}\Phi_N\\
    &+\left[C^\top\left(P_N+\frac{K_N}{N}\right)+F^\top\left(\frac{P_N}{N}+\frac{K_N}{N^2}+\frac{N-1}{N} S_N+\frac{N-1}{N^2} M_N\right)\right] g\\
	&+ \left[\widetilde{C}^\top \left(P_N+K_N\right)P_N+\widetilde{F}^\top\left(\frac{P_N}{N}+\frac{K_N}{N}+\frac{N-1}{N}\left(M_N+S_N+\Pi_N\right)\right)\right] \widetilde{g}  \\
	& \left.+\widetilde{F}^\top \left(\frac{\beta_N}{N}+\frac{N-1}{N^2} \zeta_N\right)+\widetilde{C} \beta_N\right\}{d t}+\beta_N d W_0, \\
	\varphi_N(T)= & -\left(I_n-\frac{\Gamma_0}{N}\right)^\top G \eta_0.
\end{aligned}\right.
\end{equation}
Similarly, by applying It\^o's formula to (\ref{follower decouple form 2}) and comparing the drift coefficients, we obtain the equation satisfied by the coefficients of the term $x^*_i(\cdot)$:
\begin{equation}\label{Pi_N equation}
\left\{\begin{aligned}
	& \dot{\Pi}_N+\Pi_N A+A^\top \Pi_N+\widetilde{C}^\top \Pi_N \widetilde{C}+E^\top\left(\frac{P_N}{N}+\frac{N-1}{N} \Pi_N-\frac{S_N}{N}\right) \\
	& +F^\top\left(\frac{P_N}{N}+\frac{K_N}{N^2}-\frac{S_N}{N}-\frac{M_N}{N^2}\right) C+\widetilde{F}^\top\left(\frac{P_N}{N}+\frac{N-1}{N}\Pi_N-\frac{S_N}{N}\right) \widetilde{C} \\
	&+  \left(\Pi_N B+\widetilde{C}^\top \Pi_N \widetilde{D}+F^\top \left(\frac{P_N}{N}+\frac{K_N}{N^2}-\frac{S_N}{N}-\frac{M_N}{N^2}\right) D\right.\\
	&\left.+\widetilde{F}^\top\left(\frac{P_N}{N}+\frac{N-1}{N}\Pi_N-\frac{S_N}{N}\right) \widetilde{C}\right) \mathcal{R}_N^{-1} \widetilde{P}_N-\frac{\Gamma_1^\top}{N} Q=0, \\
	&\Pi_N(T)= -\frac{\Gamma_0^\top}{N} G,
\end{aligned}\right.
\end{equation}
the equation satisfied by the coefficients of the term $x^*_j(\cdot)$:
\begin{equation}\label{S_N equation}
\left\{\begin{aligned}
	&\dot{S}_N+S_N A+A^T S_N+C^T\left(S_N+\frac{M_N}{N}\right) C+\widetilde{C}^T S_N \widetilde{C} \\
	&-\left[S_N B+C^\top\left(S_N+\frac{M_N}{N}\right) D+\widetilde{C}^\top S_N \widetilde{D}\right] \mathcal{R}_N^{-1} \widetilde{P_N}=0, \\
	&S_N(T)=0,
\end{aligned}\right.
\end{equation}
the equation satisfied by the coefficients of the term $x^{*(N)}(\cdot)$:
\begin{equation}\label{M_N equation}
\left\{\begin{aligned}
	& \dot{M}_N+A^\top M_N+M_N(A+E)+C^\top\left(S_N+\frac{M_N}{N}\right) F+\widetilde{C}^\top\left(\Pi_N+S_N\right) \widetilde{F} \\
    & +\widetilde{C}^\top M_N\left(\widetilde{C}+\widetilde{F}\right)+E^\top\left(\frac{K_N}{N}+S_N+\frac{N-1}{N} M_N\right)\\
    & +F^\top\left(\frac{P_N}{N}+\frac{K_N}{N^2}-\frac{S_N}{N}-\frac{M_N}{N^2}\right) F+F^\top\left( S_N+\frac{M_N}{N}\right)(C+F)\\
    & +\widetilde{F}^\top\left(\frac{P_N}{N}+\frac{N-1}{N} \Pi_N-\frac{S_N}{N}\right) \widetilde{F} +\widetilde{F}\left(\frac{K_n}{N}+\frac{N-1}{N} M_N+ S_N\right)\left(\widetilde{C}+\widetilde{F}\right)\\
    & +\left(\Pi_N+S_N\right) E +\left\{\left[\left(\Pi_N+S_N\right) B+C^\top\left(S_N+\frac{M_N}{N}\right) D+\widetilde{C}^\top\left(\Pi_N+S_N\right) \widetilde{D}\right.\right. \\
    & \left.+F^\top\left(\frac{P_N}{N}+\frac{K_N}{N^2}-\frac{S_N}{N}-\frac{M_N}{N^2}\right) D+\widetilde{F}^\top\left(\frac{P_N}{N}+\frac{N-1}{N} \Pi_N-\frac{S_N}{N}\right) \widetilde{D}\right\} \mathcal{R}_N^{-1} \widetilde{P}_N\\
    & -\left\{\left(\Pi_N+S_N\right) B+M_N B+C^\top\left(S_N+\frac{M_N}{N}\right) D+\widetilde{C}^\top\left(\Pi_N+S_N+M_N \right) \widetilde{D}\right. \\
    & +F^\top\left(\frac{P_N}{N}+\frac{K_N}{N^2}+\frac{N-1}{N} S_N+\frac{N-1}{N^2} M_N\right) D \\
    & \left.+\widetilde{F}_N^\top\left[\frac{P_N}{N}+\frac{K_N}{N}+\frac{N-1}{N}\left(\Pi_N+S_N+M_N\right) \widetilde{D}\right]\right\}\left(\mathcal{R}_N+\widetilde{D} K_N \widetilde{D}\right)^{-1}\\
    & \times\left(\widetilde{P_N}+\widetilde{K}_N\right)+\frac{\Gamma_1^\top}{N} Q \Gamma_1=0,\\
    &M_N(T)=\frac{\Gamma_0^\top}{N} G \Gamma_0,
\end{aligned}\right.
\end{equation}
and the equation satisfied by the non-homogeneous term:
\begin{equation}\label{psi_N equation}
\left\{\begin{aligned}
	d \psi_N  =&-\left\{\left(A^\top+\frac{N-1}{N} E^\top\right) \psi_N+\frac{E^\top}{N} \varphi_N\right. +\left(\Pi_N+S_N+M_N\right)\\
    & \times\left(f-B\left(\mathcal{R}_N+\widetilde{D} K_N \widetilde{D}\right)^{-1} \Phi_N\right)+\left[C^\top\left(S_N+\frac{M_N}{N}\right)\right.\\
	& \left.+F^\top\left(\frac{P_N}{N}+\frac{K_N}{N^2}+\frac{N-1}{N} S_N+\frac{N-1}{N^2} M_N\right)\right]\\
    & \times\left[g-D\left(\mathcal{R}_N+\widetilde{D}K_N \widetilde{D}\right)^{-1} \Phi_N\right] +\left[\widetilde{C}^\top\left(\Pi_N+S_N+M_N\right)\right.\\
	& \left.+\widetilde{F}_N^\top\left(\frac{P_N}{N}+\frac{K_N}{N}+\frac{N-1}{N}\left(\Pi_N+S_N+M_N\right)\right) \right]\\
    & \times\left[\widetilde{g}-\widetilde{D}\left(\mathcal{R}_N+\widetilde{D}K_N \widetilde{D}\right)^{-1} \Phi_N\right] \\
	& +\left(\widetilde{C}^\top+\frac{N-1}{N} \widetilde{F}^\top\right) \zeta_N+\frac{\widetilde{F}^\top}{N} \beta_N+\frac{\Gamma_1^\top}{N} Q \eta_1\bigg\}dt+\zeta_NdW_0,\\
	\psi_N(T)=&\ \frac{\Gamma_0^\top}{N} G \eta_0.
\end{aligned}\right.
\end{equation}
From the above derivation, we can further obtain the following results about the feedback representation of centralized open-loop Nash equilibrium of Problem \ref{problem centralized}.
\begin{mythm}\label{theorem feedback representation of open loop centralized}
	Let Assumptions \ref{A1}, \ref{A2}, \ref{A3} hold and the coupled Riccati equation system (\ref{P_N equation}), (\ref{K_N equation}), (\ref{Pi_N equation}), (\ref{S_N equation}), (\ref{M_N equation}) admit a solution, then centralized open-loop Nash equilibrium of Problem \ref{problem centralized} has the feedback representation
	\begin{equation}\label{optimal feedback of open loop condition of centralized}\vspace{-1mm}
    \begin{aligned}
		u_i^*=&-\mathcal{R}_N^{-1} \widetilde{P}_N x_i^*+\left[\mathcal{R}_N^{-1} \widetilde{P}_N-\left(\mathcal{R}_N+\widetilde{D}_N^\top K_N \widetilde{D}\right)^{-1}
               \left(\widetilde{P}_N+\widetilde{K}_N\right)\right] x^{*(N)}\\
              &-\left(\mathcal{R}_N+\widetilde{D} K_N \widetilde{D}\right)^{-1} \Phi_N,
	\end{aligned}
    \end{equation}
	where $x^*_i(\cdot)\in L^2_{\mathcal{F}_t}\left([0,T];\mathbb{R}^n\right)$ is the solution to
	\begin{equation}\label{x^*_i equation}
		\left\{\begin{aligned}
			d x_i^*&=\left\{\left(A-B \mathcal{R}_N^{-1} \widetilde{P}_N\right) x_i^*+\left[E+B\mathcal{R}_N^{-1} \widetilde{P}_N-B\left(\mathcal{R}_N+\widetilde{D}^\top K_N \widetilde{D}\right)^{-1}\right.\right.\\
			&\qquad \times\left.\left.\left(\widetilde{P}_N+\widetilde{K}_N\right)\right]x^{*(N)} +f-B\left(\mathcal{R}_N+\widetilde{D}^\top K_N \widetilde{D}\right)^{-1} \Phi_N\right\} d t \\
			&\quad +\left\{\left(C-D \mathcal{R}_N^{-1} \widetilde{P}_N\right) x_i^*+\left[F+D\mathcal{R}_N^{-1} \widetilde{P}_N-D\left(\mathcal{R}_N+\widetilde{D}^\top K_N \widetilde{D}\right)^{-1}\right.\right.\\
			&\qquad \times\left.\left.\left(\widetilde{P}_N+\widetilde{K}_N\right)\right] x^{*(N)} \right. \left.+g-D\left(\mathcal{R}_N+\widetilde{D}^\top K_N \widetilde{D}\right)^{-1} \Phi_N\right\} d W_i \\
			&\quad +\left\{\left(\widetilde{C}-\widetilde{D} \mathcal{R}_N^{-1} \widetilde{P}_N\right) x_i^*+\left[\widetilde{F}+\widetilde{D}\mathcal{R}_N^{-1} \widetilde{P}_N-\widetilde{D}\left(\mathcal{R}_N+\widetilde{D}^\top K_N \widetilde{D}\right)^{-1} \right.\right. \\
			&\qquad \times\left.\left.\left(\widetilde{P}_N+\widetilde{K}_N\right)\right]x^{*(N)}+\widetilde{g}-\widetilde{D}\left(\mathcal{R}_N+\widetilde{D}^\top K_N \widetilde{D}\right)^{-1} \Phi_N\right\} d W_0,\\
			x_i^*(0)&=\xi_i,
		\end{aligned}\right.
	\end{equation}
	and $x^{*(N)}(\cdot)\in L^2_{\mathcal{F}_t}\left([0,T];\mathbb{R}^n\right)$ is the solution to
	\begin{equation}\label{x^{*(N)} equation}
	\left\{\begin{aligned}
		d x^{*(N)}&=\left\{\left[A+E-\left(\mathcal{R}_N+\widetilde{D}^\top K_N \widetilde{D}\right)^{-1}\left(\widetilde{P}_N+\widetilde{K}_N\right)\right] x^{*(N)}+f\right.\\
        &\qquad \left.-B\left(\mathcal{R}_N+\widetilde{D}^\top K_N \widetilde{D}\right)^{-1} \Phi_N\right\} d t \\
		&\quad +\frac{1}{N} \sum_{j=1}^N\left\{\left(C-D \mathcal{R}_N^{-1} \widetilde{P}_N\right) x_j^*+\left[F+D\mathcal{R}_N^{-1} \widetilde{P}_N-D\left(\mathcal{R}_N+\widetilde{D}^\top K_N \widetilde{D}\right)^{-1}\right.\right. \\
		&\qquad \times\left.\left(\widetilde{P}_N+\widetilde{K}_N\right)\right] x^{*(N)}\left. +g-D\left(\mathcal{R}_N+\widetilde{D}^\top K_N \widetilde{D}\right)^{-1} \Phi_N\right\} d W_j \\
		&\quad +\left\{\left[\widetilde{C}+ \widetilde{F}-\left(\mathcal{R}_N+\widetilde{D}^\top K_N \widetilde{D}\right)^{-1}\left(\widetilde{P}_N+\widetilde{K}_N\right)\right] x^{*(N)}+\widetilde{g}\right.\\
        &\qquad \left.-\widetilde{D}\left(\mathcal{R}_N+\widetilde{D}^\top K_N \widetilde{D}\right)^{-1} \Phi_N\right\} d W_0,\\
		x^{*(N)}(0)&=\frac{1}{N}\sum_{j=1}^N \xi_j.
	\end{aligned}\right.
    \end{equation}
\end{mythm}

\subsection{Open-loop decentralized Nash equilibrium}

In this subsection, we consider the limit for the solutions of Riccati equations (\ref{P_N equation}), (\ref{K_N equation}), (\ref{Pi_N equation}), (\ref{S_N equation}), (\ref{M_N equation}) and ODEs (\ref{varphi_N equation}), (\ref{psi_N equation}) and design open-loop decentralized Nash equilibrium of Problem \ref{problem decentralized}. First, we introduce the following equations
\begin{equation}\label{P equation}
\left\{\begin{aligned}
	&\dot{P}+P A+A^\top P+C^\top P C+\widetilde{C}^\top P \widetilde{C}-\widetilde{P}^\top \mathcal{R}^{-1} \widetilde{P}+Q=0 , \\
	&P(T)=G ,
\end{aligned}\right.
\end{equation}

\begin{equation}\label{K equation}
\left\{\begin{aligned}
	&\dot{K}+K(A+E)+P E+A^\top K+C^\top P F+\widetilde{C}^\top(P+K) \widetilde{F} \\
	&+\widetilde{P}^\top \mathcal{R}^{-1} \widetilde{P}-\left(\widetilde{P}^\top+K B+\widetilde{C}^\top K \widetilde{D}\right)\left(\mathcal{R}
    +\widetilde{D}^\top K \widetilde{D}\right)^{-1}\left(\widetilde{P}+\widetilde{K}\right)-Q \Gamma_1=0,\\
	&K(T)=-G \Gamma_0,
\end{aligned}\right.
\end{equation}
\begin{equation}\label{varphi equation}
\left\{\begin{aligned}
	d \varphi= & -\bigg\{A^\top \varphi+(P+K) f-\left[(P+K) B-C^\top P D-\widetilde{C}(P+K) \widetilde{D}\right] \\
	& \left.\times\left(\mathcal{R}+\widetilde{D}^\top K \widetilde{D}\right)^{-1} \Phi+C^\top P g+\widetilde{C}^\top(P+K) \widetilde{g}+\widetilde{C} \beta-Q\eta_1\right\} d t+\beta d W_0 ,\\
	\varphi(T)= & -G \eta_0,
\end{aligned}\right.
\end{equation}

\begin{equation}\label{Pi equation}
\left\{\begin{aligned}
	&\dot{\Pi}+\Pi A+A^\top \Pi+\widetilde{C}^\top \Pi \widetilde{C}+E^\top \left(P+\Pi\right)+F^\top PC+\widetilde{F}^\top \left(P+\Pi\right) \widetilde{C}\\
	&+\left(\Pi B+\widetilde{C}^\top \Pi \widetilde{D}+F^\top PD+\widetilde{F}^\top \left(P+\Pi\right)\widetilde{D}\right) \mathcal{R}^{-1} \widetilde{P}-\Gamma_1^\top Q=0 ,\\
	&\Pi(T)=-\Gamma_0^\top G,
\end{aligned}\right.
\end{equation}
\begin{equation}\label{S equation}
\left\{\begin{aligned}
	&\dot{S}+S A+A^\top S+C^\top S C+\widetilde{C} S \widetilde{C}-\left(S B+C^\top S D+\widetilde{C}^\top S \widetilde{D}\right) R^{-1} \widetilde{P}=0 ,\\
	&S(T)=0,
\end{aligned}\right.
\end{equation}
\begin{equation}\label{M equation}
	\left\{\begin{aligned}
	&\dot{M}+A^\top M+M(A+E)+\widetilde{C}\Pi \widetilde{F}+\widetilde{C}^\top M\left(\widetilde{C}+\widetilde{F}\right)+\Pi E \\
	&+E^\top(K+M)+F^\top PF+\widetilde{F} \left(K+M\right) \widetilde{C}+\widetilde{F}^\top \left(P+\Pi+K+M\right)\widetilde{F} \\
	&+\left[\Pi B+F^\top P D+\widetilde{C}^\top\Pi \widetilde{D}+\widetilde{F}^\top \left(P+\Pi\right) \widetilde{D}\right] \mathcal{R}^{-1} \widetilde{P} \\
	&-\left[\Pi B+M B+\widetilde{C}(\Pi+M) \widetilde{D} +F^\top P D+\widetilde{F}^\top \left(P+\Pi+K+M\right)\widetilde{D}\right] \\
	&\times\left(\mathcal{R}+\widetilde{D}K \widetilde{D}\right)^{-1}\left(\widetilde{P}+\widetilde{K}\right)+\Gamma_1^\top Q\Gamma_1=0 ,\\
	&M(T)=\Gamma_0^\top G\Gamma_0,
\end{aligned}\right.
\end{equation}
\begin{equation}\label{psi equation}
\left\{\begin{aligned}
	d \psi=&-\left\{ \left(A^\top+E^\top\right) \psi+E^\top\varphi+(\Pi+M)\left(f-B(\mathcal{R}+\widetilde{D} K \widetilde{D})^{-1} \Phi\right)\right. \\
	& +F^\top P\left[g-D\left(\mathcal{R}+\widetilde{D} K \widetilde{D}\right)^{-1} \Phi\right] \\
	&+\left[\widetilde{C}^\top(\Pi+M)+\widetilde{F}^\top(P+K+\Pi+M)\right]\left[\widetilde{g}-\widetilde{D}\left(\mathcal{R}+\widetilde{D} K \widetilde{D}\right)^{-1} \Phi\right] \\
	& \left.+\left(\widetilde{C}^\top+\widetilde{F}^\top\right) \zeta+\widetilde{F}^\top\beta+\Gamma_1^\top Q\eta_1\right\} d t+\zeta d W_0, \\
	\psi(T)&=\Gamma_0^\top G\eta_0,
\end{aligned}\right.
\end{equation}
where
$$
\begin{aligned}
	& \mathcal{R}:=R+D^\top P D+\widetilde{D}^\top P \widetilde{D}, \\
	& \widetilde{P}:=B^\top P+D^\top P C+\widetilde{D}^\top P \widetilde{C} ,\\
	& \widetilde{K}:=B^\top K+D^\top P F+\widetilde{D}^\top P \widetilde{F}+\widetilde{D}^\top K\left(\widetilde{C}+\widetilde{F}\right),\\
	& \Phi:=B^\top \varphi+D^\top P g+\widetilde{D}^\top\left(P+K\right) \widetilde{g}+\widetilde{D}^\top \beta-R\eta_2.
\end{aligned}
$$
And we need the following assumptions.
\begin{myassump}\label{A4}
	$$
	\begin{aligned}
		& R+D^\top P D+\widetilde{D}^\top P \widetilde{D}>0,\\
		& R+D^\top P D+\widetilde{D}^\top (P+K) \widetilde{D}>0.
	\end{aligned}
	$$
\end{myassump}

\begin{myassump}\label{A5}
	The equation (\ref{K equation}) admits unique solution in $C^1\left([0,T];\mathbb{R}^{n\times n}\right)$.
\end{myassump}

Then, we have the following result.
\begin{mythm}
	Let Assumptions \ref{A1}, \ref{A2}, \ref{A3}, \ref{A4}, \ref{A5} hold, then the equations (\ref{P equation}), (\ref{K equation}), (\ref{Pi equation}), (\ref{S equation}), (\ref{M equation}) admits unique solution in $C\left([0,T],\mathbb{R}^{n\times n}\right)$ and the equations (\ref{varphi equation}), (\ref{psi equation}) admits unique solution in $L^2_{\mathcal{F}_t}\left([0,T],\mathbb{R}^{n}\right)$. In addition, the equation (\ref{P_N equation}), (\ref{K_N equation}), (\ref{Pi_N equation}), (\ref{S_N equation}), (\ref{M_N equation}) have asymptotic solvability and the following estimations:
	\begin{equation}\label{estimation of P,K}
		\sup_{t\in[0,T]}\Big(\left|P_N(t)-P(t)\right|+\left|K_N(t)-K(t)\right|\Big)=O\left(\frac{1}{N}\right),
	\end{equation}
	\begin{equation}\label{estimation of Pi,S,M}
		\sup_{t\in[0,T]}\Big(\left|N\Pi_N(t)-\Pi(t)\right|+\left|NS_N(t)-0\right|+\left|NM_N(t)-M(t)\right|\Big)=O\left(\frac{1}{N}\right),
	\end{equation}
	\begin{equation}\label{estimation of varphi, beta}
		\sup_{t\in[0,T]}\mathbb{E}\left|\varphi_N(t)-\varphi(t)\right|^2+\sup_{t\in[0,T]}\mathbb{E}\left|\beta_N(t)-\beta(t)\right|^2=O\left(\frac{1}{N^2}\right),
	\end{equation}
	\begin{equation}\label{estimation of varphi, zeta}
		\sup_{t\in[0,T]}\mathbb{E}\left|N\psi_N(t)-\psi(t)\right|^2+\sup_{t\in[0,T]}\mathbb{E}\left|N\zeta_N(t)-\zeta(t)\right|^2=O\left(\frac{1}{N^2}\right).
	\end{equation}
\end{mythm}
\begin{proof}
	Noting that equation (\ref{P equation}) is a standard Riccati equation, we can employ Lemma 3.1 from \cite{Li-Nie-Wang-Yan-2024} to know that it admits a unique solution in $C^1\left([0,T];\mathbb{R}^{n\times n}\right)$. Moreover, we observe that equations (\ref{Pi equation})-(\ref{M equation}) are all linear {\it ordinary differential equations} (ODEs). Therefore, under the boundedness condition of Assumption \ref{A1} and Assumption \ref{A5}, equations (\ref{Pi equation})-(\ref{M equation}) admit unique solutions. In particular, $S(\cdot)=0$ is the unique solution of equation (\ref{S equation}). Since both (\ref{varphi equation}) and (\ref{psi equation}) are linear BSDEs, under Assumption \ref{A1}, they both admit unique solutions. Using the continuous dependence of the solutions to ODEs on parameters (see Theorem 4 in \cite{Huang-Zhou-2020}), we have the estimation (\ref{estimation of P,K}), (\ref{estimation of Pi,S,M}). Using the standard estimates for BSDE (see Theorem 7.2.2 in \cite{Yong-Zhou-1999}), we have the estimations (\ref{estimation of varphi, beta}) and (\ref{estimation of varphi, zeta}). Then, the proof is completed.
\end{proof}

Here, we provide two results regarding the solvability of equation (\ref{K equation}). First, let $N(\cdot):=P(\cdot)+K(\cdot)$, we have the following equation
\begin{equation}\label{N equation}
\left\{\begin{aligned}
	& \dot{N}+N(A+E)+A^\top N+C^\top P C+\widetilde{C}^\top N \widetilde{C}+C^\top P F+\widetilde{C}^\top N \widetilde{F} \\
	& +Q-Q \Gamma_1-\left[N B+\widetilde{C}^\top N \widetilde{D}+C^\top P D\right]\left(R+D^\top P D+\widetilde{D}^\top N \widetilde{D}\right)^{-1} \\
	& \times\left[B^\top N+D^\top P(C+F)+\widetilde{D}^\top N\left(\widetilde{C}^\top+\widetilde{F}\right)\right]=0, \\
	& N (T)=G-G \Gamma_0.
\end{aligned}\right.
\end{equation}
And let $\Psi(\cdot)$ be the fundamental solution matrix of the following linear ODE:
$$
\begin{aligned}
&\begin{pmatrix}
	\dot{u} \\
	\dot{v}
\end{pmatrix}
=H\begin{pmatrix}
	u \\
	v
\end{pmatrix}dt,\quad
\begin{pmatrix}
	u \\
	v
\end{pmatrix}(T)=\begin{pmatrix}
I_n \\
I_n
\end{pmatrix},
\end{aligned}
$$
where $H:=\begin{pmatrix}
	a_1 & a_2  \\
	a_3 & a_4
\end{pmatrix}$, with
$$
\begin{aligned}
	&a_1:=A+E-B R^{-1} D^\top P(C+F),\\
	&a_2:= -B\left(R+D^\top P D\right)^{-1} B^\top,\\
	&a_3:=-C^\top P(C+F)-\widetilde{C}^\top P \widetilde{C}+C^\top P D R^{-1} D^\top P(C+F)-Q+Q\Gamma_1,\\
	&a_4:=-A^\top-\delta_1\delta_2I_n+C^\top P D
	\left(R+D^\top P D\right)^{-1} B^\top-\delta_2^2I_n,
\end{aligned}
$$
where constants $\delta_i>0,i=1,2$. We first have the following result.
\begin{mypro}
	If $\widetilde{D}(\cdot)\equiv0$, $\widetilde{F}(\cdot)=\delta_1I_n$, $\widetilde{C}(\cdot)=\delta_2I_n$, and the matrix $
	\begin{pmatrix}
		I_n & 0
	\end{pmatrix} \Psi(t) \Psi^{-1}(T)\\\begin{pmatrix}
		I_n \\
		G-G \Gamma_0
	\end{pmatrix}
	$ is non-singular, then the Assumption \ref{A4} holds.
\end{mypro}
\begin{proof}
	Under the conditions of the proposition, equation (\ref{N equation}) becomes
	\begin{equation}\label{N_1 equation}
		\left\{\begin{aligned}
			& \dot{N}+N\left[A+E-B\left(R+D^\top P D\right)^{-1} D^\top P(C+F)\right] \\
			& +\left[A^\top+\delta_1\delta_2 I_n+\delta^2_2I_n-C^\top P D\left(R+D^\top P D\right)^{-1} B^\top\right] N \\
			& -N B\left(R+D^\top P D\right)^{-1} B^\top N+C^\top P C+\widetilde{C}^\top P \widetilde{C}+C^\top P F \\
			& +Q-Q \Gamma_1-C^\top P D\left(R+D^\top P D\right)^{-1} D^\top P(C+F) =0,\\
			& N (T)=G-G \Gamma_0.
		\end{aligned}\right.
	\end{equation}
	By the result of Page 11 in Reid \cite{Reid-1972}, we have the necessary and sufficient condition for the wellposedness of (\ref{N_1 equation}) is the matrix $
	\begin{pmatrix}
		I_n & 0
	\end{pmatrix} \Psi(t) \Psi^{-1}(T)\begin{pmatrix}
		I_n \\
		G-G \Gamma_0
	\end{pmatrix}
	$ is non-singular. Then $K=N-P$ is the unique solution to equation (\ref{K equation}). The proof is complete.
\end{proof}
Next, we present another result about Assumption \ref{A4}.
\begin{mypro}
	If $F(\cdot)=\widetilde{F}(\cdot)\equiv0$, $E(\cdot)=\delta_3I_n$, $\Gamma_1(\cdot)=\delta_4I_n$, and $\Gamma_0=\delta_5I_n$, $0 <\delta_3,\delta_4,\delta_5 \leq 1$, then Assumption \ref{A4} holds.
\end{mypro}
\begin{proof}
	Under the conditions of the proposition, equation (\ref{N equation}) becomes
	\begin{equation}\label{N_2 equation}
		\left\{\begin{aligned}
			&\dot{N}+N A+A^\top N+\delta_3 N+C^\top P C+\widetilde{C}^\top N \widetilde{C}+Q-\delta_4 Q \\
			&-\left(N B+\widetilde{C}^\top N \widetilde{D} +C^\top P D\right)\left(R+D^\top P D+\widetilde{D}^\top N \widetilde{D}\right)^{-1}\left[B^\top N+D^\top P C+\widetilde{D}^\top N \widetilde{C}\right] = 0,\\
			&N(T)=G-\delta_5 G.
		\end{aligned}\right.
	\end{equation}
	Noting $Q-\delta_4 Q\geq0$ and $G-\delta_5 G\geq0$, we find that equation (\ref{N_2 equation}) is a standard Riccati equation. By using Theorem 7.2, Chapter 6 of \cite{Yong-Zhou-1999}, we obtain the existence and uniqueness of the solution to equation (\ref{N_2 equation}). Then $K=N-P$ is the unique solution to equation (\ref{K equation}).
\end{proof}

Next, we construct open-loop decentralized Nash equilibrium of Problem \ref{problem decentralized}.

\begin{mythm}\label{theorem feedback representation of open loop decentralized}
	Let Assumptions \ref{A1}, \ref{A2}, \ref{A3}, \ref{A4}, \ref{A5} hold, then the decentralized open-loop Nash equilibrium strategy of Problem \ref{problem decentralized} is
	\begin{equation}\label{optimal feedback of open loop condition of decentralized}
		\bar{u}_i^*=-\mathcal{R}^{-1} \widetilde{P} \bar{x}_i^*+\left[\mathcal{R}^{-1} \widetilde{P}-\left(\mathcal{R}+\widetilde{D}^\top K \widetilde{D}\right)^{-1}\left(\widetilde{P}+\widetilde{K}\right)\right] \bar{x}^*-\left(\mathcal{R}+\widetilde{D}^\top K \widetilde{D}\right)^{-1} \Phi,
	\end{equation}
	where $\bar{x}^*(\cdot)\in L^2_{\mathcal{F}^0_t}\left([0,T];\mathbb{R}^n\right)$ is the solution to
	\begin{equation}\label{bar{x}^* equation}
		\left\{\begin{aligned}
			d \bar{x}^*=&\left\{\left[A+E-\left(\mathcal{R}+\widetilde{D}^\top K \widetilde{D}\right)^{-1}(\widetilde{P}+\widetilde{K})\right] \bar{x}^*+f-B\left(\mathcal{R}+\widetilde{D}^\top K \widetilde{D}\right)^{-1} \Phi\right\} d t \\
			&+\left\{\left[\widetilde{C}+\widetilde{F}-\left(\mathcal{R}+\widetilde{D}^\top K \widetilde{D}\right)^{-1}(\widetilde{P}+\widetilde{K})\right] \bar{x}^*+\widetilde{g}-\widetilde{D}\left(\mathcal{R}+\widetilde{D}^\top K \widetilde{D}\right)^{-1} \Phi\right\} d W_0,\\
			\bar{x}^*(0)&=\bar{\xi},
		\end{aligned}\right.
	\end{equation}
	and $\bar{x}_i^*(\cdot)\in L^2_{\mathcal{F}^i_t}\left([0,T];\mathbb{R}^n\right)$ is the solution to
	\begin{equation}\label{bar{x}_i^* equation}
		\left\{\begin{aligned}
			d \bar{x}_i^*=&\left\{\left(A-B \mathcal{R}^{-1} \widetilde{P}\right) \bar{x}_i^*+\left[E+B\mathcal{R}^{-1} \widetilde{P}-B\left(\mathcal{R}+\widetilde{D}^\top K \widetilde{D}\right)^{-1}\left(\widetilde{P}+\widetilde{K}\right)\right] \bar{x}^*\right. \\
			&\quad +f-B\left(\mathcal{R}+\widetilde{D}^\top K \widetilde{D}\right)^{-1} \Phi\bigg\} d t \\
			&+\left\{\left(C-D \mathcal{R}^{-1} \widetilde{P}\right) \bar{x}_i^*+\left[F+D\mathcal{R}^{-1} \widetilde{P}-D\left(\mathcal{R}+\widetilde{D}^\top K \widetilde{D}\right)^{-1}\left(\widetilde{P}+\widetilde{K}\right)\right] \bar{x}^*\right. \\
			&\quad +g-D\left(\mathcal{R}+\widetilde{D}^\top K \widetilde{D}\right)^{-1} \Phi\bigg\} d W_i \\
			&+\left\{\left(\widetilde{C}-\widetilde{D} \mathcal{R}^{-1} \widetilde{P}\right) \bar{x}_i^*+\left[\widetilde{F}+\widetilde{D}\mathcal{R}^{-1} \widetilde{P}-\widetilde{D}\left(\mathcal{R}+\widetilde{D}^\top K \widetilde{D}\right)^{-1}\left(\widetilde{P}+\widetilde{K}\right)\right] \bar{x}^*\right. \\
			&\quad +\widetilde{g}-\widetilde{D}\left(\mathcal{R}+\widetilde{D}^\top K \widetilde{D}\right)^{-1} \Phi\bigg\} d W_0, \\
			\bar{x}_i^*(0)&=\xi_i.
		\end{aligned}\right.
	\end{equation}
	And the optimal state $\hat{x}^*_i(\cdot)\in L^2_{\mathcal{F}_t}\left([0,T];\mathbb{R}^n\right)$ corresponding to decentralized open-loop Nash equilibrium strategy is the solution to
	\begin{equation}\label{hat{x}_i^* equation}
		\left\{\begin{aligned}
			d \hat{x}_i^*=	&\left\{A\hat{x}_i^*-B \mathcal{R}^{-1} \widetilde{P}\bar{x}_i^*+\left[B\mathcal{R}^{-1} \widetilde{P}-B\left(\mathcal{R}+\widetilde{D}^\top K \widetilde{D}\right)^{-1}\left(\widetilde{P}+\widetilde{K}\right)\right] \bar{x}^*\right. \\
			&\quad -B\left(\mathcal{R}+\widetilde{D}^\top K \widetilde{D}\right)^{-1} \Phi+E\hat{x}^{*(N)}+f\bigg\} d t \\
			&+\left\{C\hat{x}_i^*-D \mathcal{R}^{-1} \widetilde{P} \bar{x}_i^*+\left[D\mathcal{R}^{-1} \widetilde{P}-D\left(\mathcal{R}+\widetilde{D}^\top K \widetilde{D}\right)^{-1}\left(\widetilde{P}+\widetilde{K}\right)\right] \bar{x}^*\right. \\
			&\quad -D\left(\mathcal{R}+\widetilde{D}^\top K \widetilde{D}\right)^{-1} \Phi+F\hat{x}^{*(N)}+g\bigg\} d W_i \\
			&+\left\{\widetilde{C}\hat{x}_i^*-\widetilde{D} \mathcal{R}^{-1} \widetilde{P} \bar{x}_i^*+\left[\widetilde{D}\mathcal{R}^{-1} \widetilde{P}-\widetilde{D}\left(\mathcal{R}+\widetilde{D}^\top K \widetilde{D}\right)^{-1}\left(\widetilde{P}+\widetilde{K}\right)\right] \bar{x}^*\right. \\
			&\quad -\widetilde{D}\left(\mathcal{R}+\widetilde{D}^\top K \widetilde{D}\right)^{-1} \Phi+\widetilde{F}\hat{x}^{*(N)}+\widetilde{g}\bigg\} d W_0, \\
			\hat{x}_i^*(0)&=\xi_i,
		\end{aligned}\right.
	\end{equation}
	the optimal state average term $\hat{x}^{*(N)}(\cdot)\in L^2_{\mathcal{F}_t}\left([0,T];\mathbb{R}^n\right)$ corresponding to decentralized open-loop Nash equilibrium strategy is the solution to
	\begin{equation}\label{hat{x}^{*(N)} equation}
		\left\{\begin{aligned}
			d \hat{x}^{*(N)}=&\left\{\left(A+E\right)\hat{x}^{*(N)}-B \mathcal{R}^{-1} \widetilde{P} \bar{x}^{*(N)}+\left[B\mathcal{R}^{-1} \widetilde{P}-B\left(\mathcal{R}+\widetilde{D}^\top K \widetilde{D}\right)^{-1}\right.\right. \\
			&\quad \times\left.\left(\widetilde{P}+\widetilde{K}\right)\right] \bar{x}^*-B\left(\mathcal{R}+\widetilde{D}^\top K \widetilde{D}\right)^{-1} \Phi+f\bigg\} d t \\
			&+\frac{1}{N}\sum_{j=1}^{N}\left\{C\hat{x}_j^*-D \mathcal{R}^{-1} \widetilde{P} \bar{x}_j^*+\left[D\mathcal{R}^{-1} \widetilde{P}-D\left(\mathcal{R}+\widetilde{D}^\top K \widetilde{D}\right)^{-1}\right.\right. \\
			&\quad \times\left.\left(\widetilde{P}+\widetilde{K}\right)\right] \bar{x}^*-D\left(\mathcal{R}+\widetilde{D}^\top K \widetilde{D}\right)^{-1} \Phi+F\hat{x}^{*(N)}+g\bigg\} d W_j\\
			&+\left\{\left(\widetilde{C}+\widetilde{F}\right)\hat{x}^{*(N)}-\widetilde{D} \mathcal{R}^{-1} \widetilde{P} \bar{x}^{*(N)}+\left[\widetilde{D}\mathcal{R}^{-1} \widetilde{P}-\widetilde{D}\left(\mathcal{R}+\widetilde{D}^\top K \widetilde{D}\right)^{-1}\right.\right. \\
			&\quad \times\left.\left(\widetilde{P}+\widetilde{K}\right)\right] \bar{x}^*-\widetilde{D}\left(\mathcal{R}+\widetilde{D}^\top K \widetilde{D}\right)^{-1} \Phi+\widetilde{g}\bigg\} d W_0, \\
			 \hat{x}^{*(N)}(0)&=\frac{1}{N}\sum_{j=1}^{N}\xi_j.
		\end{aligned}\right.
	\end{equation}
\end{mythm}

\begin{proof}
	First, we need to prove the following estimates:
	\begin{equation}\label{CC1}
		\sup _{0 \leq t \leq T} \mathbb{E}\left[\left|\bar{x}^{*(N)}(t)-\bar{x}^*(t)\right|^2\right] = O\left(\frac{1}{N}\right),
	\end{equation}
	\begin{equation}\label{CC2}
		\sup _{0 \leq t \leq T} \mathbb{E}\left[\left|\hat{x}^{*(N)}(t)-\bar{x}^*(t)\right|^2\right] = O\left(\frac{1}{N}\right),
	\end{equation}
	\begin{equation}\label{CC3}
		\sup _{0 \leq t \leq T} \mathbb{E}\left[\left|\hat{x}_i^{*}(t)-\bar{x}_i^*(t)\right|^2\right] = O\left(\frac{1}{N}\right).
	\end{equation}
	In fact, we can get
	\begin{equation}\label{bar{x}^{*(N)} equation}
		\left\{\begin{aligned}
			d \bar{x}^{*(N)}=	&\left\{\left(A-B \mathcal{R}^{-1} \widetilde{P}\right) \bar{x}^{*(N)}+\left[E+B\mathcal{R}^{-1} \widetilde{P}-B\left(\mathcal{R}+\widetilde{D}^\top K \widetilde{D}\right)^{-1}\right.\right. \\
			&\quad \left.\times\left(\widetilde{P}+\widetilde{K}\right)\right] \bar{x}^*+f-B\left(\mathcal{R}+\widetilde{D}^\top K \widetilde{D}\right)^{-1} \Phi\bigg\} d t \\
			&+\frac{1}{N}\sum_{j=1}^{N}\left\{\left(C-D \mathcal{R}^{-1} \widetilde{P}\right) \bar{x}_j^*+\left[F+D\mathcal{R}^{-1} \widetilde{P}-D\left(\mathcal{R}+\widetilde{D}^\top K \widetilde{D}\right)^{-1}\right.\right. \\
			&\quad \left.\times\left(\widetilde{P}+\widetilde{K}\right)\right] \bar{x}^*+g-D\left(\mathcal{R}+\widetilde{D}^\top K \widetilde{D}\right)^{-1} \Phi\bigg\} d W_j \\
			&+\left\{\left(\widetilde{C}-\widetilde{D} \mathcal{R}^{-1} \widetilde{P}\right) \bar{x}^{*(N)}+\left[\widetilde{F}+\widetilde{D}\mathcal{R}^{-1} \widetilde{P}-\widetilde{D}\left(\mathcal{R}+\widetilde{D}^\top K \widetilde{D}\right)^{-1}\right.\right. \\
			&\quad \left.\times\left(\widetilde{P}+\widetilde{K}\right)\right] \bar{x}^*+\widetilde{g}-\widetilde{D}\left(\mathcal{R}+\widetilde{D}^\top K \widetilde{D}\right)^{-1} \Phi\bigg\} d W_0, \\
			\bar{x}^{*(N)}(0)&=\frac{1}{N}\sum_{j=1}^{N}\xi_j.
		\end{aligned}\right.
	\end{equation}
	From (\ref{bar{x}^* equation}) and (\ref{bar{x}^{*(N)} equation}), we have
	\begin{equation*}
		\left\{\begin{aligned}
			d\left(\bar{x}^{*(N)}-\bar{x}^*\right)  =&\left(A-B \mathcal{R}^{-1}\widetilde{P}\right)\left(\bar{x}^{*(N)}-\bar{x}^*\right) d t
+\frac{1}{N}\sum_{j=1}^{N}\bigg\{\left(C-D \mathcal{R}^{-1} \widetilde{P}\right) \bar{x}_j^*\\
			&\quad \left.+\left[F+D\mathcal{R}^{-1} \widetilde{P}-D\left(\mathcal{R}+\widetilde{D}^\top K \widetilde{D}\right)^{-1}\left(\widetilde{P}+\widetilde{K}\right)\right] \bar{x}^*\right. \\
			&\quad +g-D(\mathcal{R}+\widetilde{D}^\top K \widetilde{D})^{-1} \Phi\bigg\} d W_j +\left(\widetilde{C}-\widetilde{D} \mathcal{R}^{-1} \widetilde{P}\right)\left(\bar{x}^{*(N)}-\bar{x}^*\right) d W_0 ,\\
			\left(\bar{x}^{*(N)}-\bar{x}^*\right)(0) =&\ \frac{1}{N}\sum_{j=1}^{N}\xi_j-\bar{\xi}.
		\end{aligned}\right.
	\end{equation*}
	By using BDG's inequality and noting the boundedness condition of the coefficients and independence between $\xi_i$ in Assumption \ref{A1}, we have ($K$ is some positive constant and may be different line by line)
	\begin{equation*}
		\begin{aligned}
			& \mathbb{E}\left[\left|\bar{x}^{*(N)}(t)-\bar{x}^*(t)\right|^2\right]\\
			&\leq  K\left[\mathbb{E}\left(\frac{1}{N}\sum_{j=1}^{N}\xi_j-\bar{\xi}\right)^2+ \mathbb{E} \int_{0}^{t}\left|\bar{x}^{*(N)}-\bar{x}^*\right|ds+\frac{1}{N^2}\mathbb{E}\int_0^t \sum_{j=1}^N\left|\bar{x}_j^*+\bar{x}^*+1\right|^2 d s\right]\\
			&\leq K\left[\frac{1}{N^2}\sum_{j=1}^{N}\mathbb{E}\left(\xi_j-\bar{\xi}\right)^2+ \mathbb{E} \int_{0}^{t}\left|\bar{x}^{*(N)}-\bar{x}^*\right|ds+\frac{1}{N^2}\mathbb{E}\int_0^t \sum_{j=1}^N\left|\bar{x}_j^*+\bar{x}^*+1\right|^2 d s\right]\\
			&=K\left[ \mathbb{E} \int_{0}^{t}\left|\bar{x}^{*(N)}-\bar{x}^*\right|ds+O\left(\frac{1}{N}\right)\right].
		\end{aligned}
	\end{equation*}
	By applying Gronwall's inequality, (\ref{CC1}) holds. From (\ref{bar{x}^* equation}) and (\ref{hat{x}^{*(N)} equation}), we have
	\begin{equation*}
		\left\{\begin{aligned}
			d\left(\hat{x}^{*(N)}-\bar{x}^*\right)  &=\left[\left(A+E\right)\left(\hat{x}^{*(N)}-\bar{x}^*\right)-B \mathcal{R}^{-1}\widetilde{P}\left(\bar{x}^{*(N)}-\bar{x}^*\right)\right] d t \\
			&\quad +\frac{1}{N}\sum_{j=1}^{N}\left\{C\hat{x}_j^*-D \mathcal{R}^{-1}\widetilde{P} \bar{x}_j^*+\left[D\mathcal{R}^{-1}\widetilde{P}-D\left(\mathcal{R}+\widetilde{D}^\top K \widetilde{D}\right)^{-1}\right.\right. \\
			&\qquad \left.\times\left(\widetilde{P}+\widetilde{K}\right)\right] \bar{x}^*-D(\mathcal{R}+\widetilde{D}^\top K \widetilde{D})^{-1} \Phi+F\hat{x}^{*(N)}+g\bigg\} d W_j\\
			&\quad +\left[\left(\widetilde{C}+\widetilde{F}\right)\left(\hat{x}^{*(N)}-\bar{x}^*\right)-\widetilde{D} \mathcal{R}^{-1}\widetilde{P}\left(\bar{x}^{*(N)}-\bar{x}^*\right)\right] dW_0,\\
			\left(\bar{x}^{*(N)}-\bar{x}^*\right)(0) &=\frac{1}{N}\sum_{j=1}^{N}\xi_j-\bar{\xi}.
		\end{aligned}\right.
	\end{equation*}
	Similarly, we have
	\begin{equation*}
		\begin{aligned}
			& \mathbb{E}\left[\left|\hat{x}^{*(N)}(t)-\bar{x}^*(t)\right|^2\right]\leq  K\Bigg[\mathbb{E}\left(\frac{1}{N}\sum_{j=1}^{N}\xi_j-\bar{\xi}\right)^2 + \mathbb{E} \int_{0}^{t}\left|\hat{x}^{*(N)}-\bar{x}^*\right|ds\\
			&\quad + \mathbb{E} \int_{0}^{t}\left|\bar{x}^{*(N)}-\bar{x}^*\right|ds+\frac{1}{N^2}\mathbb{E}\int_0^t \sum_{j=1}^N\left|\bar{x}_j^*+\bar{x}^*+1\right|^2 d s\Bigg]\\
			&\leq K\left[\frac{1}{N^2}\sum_{j=1}^{N}\mathbb{E}\left|\xi_j-\bar{\xi}\right|^2+ \mathbb{E} \int_{0}^{t}\left|\hat{x}^{*(N)}-\bar{x}^*\right|ds+ \mathbb{E} \int_{0}^{t}\left|\bar{x}^{*(N)}-\bar{x}^*\right|ds\right.\\
			&\qquad \left.+\frac{1}{N^2}\mathbb{E}\int_0^t \sum_{j=1}^N\left|\bar{x}_j^*+\bar{x}^*+1\right|^2 d s\right]\\
			&=K\left[ \mathbb{E} \int_{0}^{t}\left|\bar{x}^{*(N)}-\bar{x}^*\right|ds+ \mathbb{E} \int_{0}^{t}\left|\bar{x}^{*(N)}-\bar{x}^*\right|ds+O\left(\frac{1}{N}\right)\right].
		\end{aligned}
	\end{equation*}
	By applying the estimations (\ref{CC1}) and Gronwall's inequality, (\ref{CC2}) holds.
	
	Similarly, using the estimation (\ref{CC1}) and (\ref{CC2}), we have the estimation (\ref{CC3}).
	
	Next, we prove the asymptotic optimality of the decentralized strategies, i.e., for any $i=1,\cdots,N$, we have
	\begin{equation*}
		\mathcal{J}_i\left(u^*_i(\cdot), u^*_{-i}(\cdot\right)\leq\inf _{u_i(\cdot) \in\, \mathcal{U}_i^{c}} \mathcal{J}_i\left(u_i(\cdot), u^*_{-i}(\cdot)\right)+O\left(\frac{1}{N}\right).
	\end{equation*}
	For given $i$, we only need to consider $u_i(\cdot)$ satisfies $\mathcal{J}_i\left(u_i(\cdot), u^*_{-i}(\cdot)\right) \leq \mathcal{J}_i\left(u^*_i(\cdot), u^*_{-i}(\cdot)\right)$. And then, using Assumption \ref{A2}, there exists $\delta_0>0$ such that
	$$
	\delta_0 \mathbb{E} \int_0^T  u_i^2(t) d t \leq \mathcal{J}_i\left(u_i(\cdot), u^*_{-i}(\cdot)\right) \leq \mathcal{J}_i\left(u^*_i(\cdot), u^*_{-i}(\cdot)\right),
	$$
	i.e.,
	$$
	\mathbb{E} \int_0^T\left|u_i(t)\right|^2 d t \leq K .
	$$
	We denote
	$\widetilde{u}_i(\cdot):=u_i(\cdot)-u^*_i(\cdot) \in \mathcal{U}^c_i$. We define $\left(\widetilde{x}_1(\cdot),\cdots,\widetilde{x}_N(\cdot)\right)$ as the state of all agents corresponding to strategy $\left(u^*_1(\cdot),\cdots, u^*_{i-1}(\cdot),u_i(\cdot),u^*_{i+1}(\cdot),\cdots,u^*_N(\cdot)\right)$. Then, $\widetilde{x}_i(\cdot)$ satisfies
	$$
	\left\{\begin{aligned}
		d\widetilde{x}_i=&\left(A\widetilde{x}_i+B \widetilde{u}_i+E\widetilde{x}^{(N)}\right) d t+  \left(C\widetilde{x}_i+D \widetilde{u}_i+F \widetilde{x}^{(N)}\right) d W_i \\
		&+\left(\widetilde{C}\widetilde{x}_i+\widetilde{D} \widetilde{u}_i+\widetilde{F} \widetilde{x}^{(N)}\right) d W_0, \\
		\widetilde{x}_i(0)=&\ 0,
	\end{aligned}\right.
	$$
	and $\widetilde{x}_j(\cdot)$, $j\neq i$ satisfies
	$$
	\left\{\begin{aligned}
		d\widetilde{x}_j=&\left(A\widetilde{x}_j+E\widetilde{x}^{(N)}\right) d t+  \left(C\widetilde{x}_j+F \widetilde{x}^{(N)}\right) d W_j+\left(\widetilde{C}\widetilde{x}_j+\widetilde{F} \widetilde{x}^{(N)}\right) d W_0, \\
		\widetilde{x}_j(0)=&\ 0,
	\end{aligned}\right.
	$$
	where $\widetilde{x}^{(N)}:=\frac{1}{N}\sum_{j=1}^{N}\widetilde{x}_j$. Thus,
	$$
	\begin{aligned}
		&\mathcal{J}_i\left(u_i(\cdot), u^*_{-i}(\cdot)\right)-\mathcal{J}_i\left(u^*_i(\cdot), u^*_{-i}(\cdot\right)\\
		=&\ \frac{1}{2} \mathbb{E} \int_0^T \left(\left\langle Q\left(\widetilde{x}_i-\Gamma_1 \widetilde{x}^{(N)}\right), \widetilde{x}_i-\Gamma_i \widetilde{x}^{(N)}\right\rangle+\left\langle R \widetilde{u}_i, \widetilde{u}_i\right\rangle\right) d t \\
		& +\frac{1}{2}\mathbb{E}\left\langle G\left(\widetilde{x}_i(T)-\Gamma_0 \widetilde{x}^{(N)}\right), \widetilde{x}_i(T)-\Gamma_0 \widetilde{x}^{(N)}(T)\right\rangle\\
		&+ \mathbb{E} \int_0^T\left(\left\langle Q\left(\widetilde{x}_i-\Gamma_1 \widetilde{x}^{(N)}\right), \hat{x}_i^*-\Gamma_1 \hat{x}^{*(N)}-\eta_1\right\rangle+\left\langle R \widetilde{u}_i, u_i^*-\eta_2\right\rangle\right) d t \\
		&+ \mathbb{E}\left\langle G\left(\widetilde{x}_i(T)-\Gamma_0 \widetilde{x}^{(N)}\right), \hat{x}_i^*(T)-\Gamma_0 \hat{x}^{*(N)}(T)-\eta_0\right\rangle
		:= \frac{1}{2}\widetilde{\mathcal{J}}_i\left(\widetilde{u}_i(\cdot)\right)+ I_i.
	\end{aligned}
	$$
	Due to Assumption \ref{A2}, we have $\widetilde{\mathcal{J}}_i\left(\widetilde{u}_i(\cdot)\right)>0$, then $\mathcal{J}_i\left(u_i(\cdot), u^*_{-i}(\cdot)\right)-\mathcal{J}_i\left(u^*_i(\cdot), u^*_{-i}\left(\cdot\right)\right)>I_i$. Therefore, we only need to prove that $I_i=O\left(\frac{1}{N}\right)$. In fact,
	\begin{equation}\label{I_i 1}
		\begin{aligned}
			I_i=&\ \mathbb{E} \left[\int_0^T\left(\left\langle Q\left(\widetilde{x}_i-\Gamma_1\widetilde{x}^{(N)}\right), \bar{x}_i^*-\Gamma_1 \bar{x}^{*}-\eta_1\right\rangle+\left\langle R \widetilde{u}_i, u_i^*-\eta_2\right\rangle\right.\right. \\
			&\qquad \left.+ \left\langle Q\left(\widetilde{x}_i-\Gamma_1 \widetilde{x}^{(N)}\right), \hat{x}_i^*-\bar{x}_i^*-\Gamma_1 \left(\hat{x}^{*(N)}-\bar{x}^{*}\right)\right\rangle\right) d t\\
			&\qquad +\left\langle G\left(\widetilde{x}_i(T)-\Gamma_0 \widetilde{x}^{(N)}\right), \bar{x}_i^*(T)-\Gamma_0 \bar{x}^{*}(T)-\eta_0\right\rangle\\
			&\qquad +\left\langle G\left(\widetilde{x}_i(T)-\Gamma_0 \widetilde{x}^{(N)}(T)\right), \hat{x}_i^*(T)-\bar{x}_i^*(T)-\Gamma_0 \left(\hat{x}^{*(N)}(T)-\bar{x}^{*}(T)\right)\right\rangle\bigg].
		\end{aligned}
	\end{equation}
	Let
	\begin{equation}\label{lambda i}
		\lambda_i(\cdot):=P(\cdot) \bar{x}_i^*(\cdot)+K(\cdot) \bar{x}^*(\cdot)+\varphi(\cdot),
	\end{equation}
	\begin{equation}\label{lambda j}
		\lambda^i_j(\cdot):=\Pi(\cdot) \bar{x}_i^*(\cdot)+M(\cdot) \bar{x}^*(\cdot)+\psi(\cdot),\quad j=1,\cdots,N,
	\end{equation}
	$$
	\begin{aligned}
		\theta(\cdot)&:=\Pi(\cdot)\left(\widetilde{C}(\cdot) \bar{x}^*_i(\cdot)+\widetilde{D}(\cdot) u_i^*(\cdot)+\widetilde{F}(\cdot)\bar{x}^*(\cdot)+\widetilde{g}(\cdot)\right)\\ 
                     &\qquad +M(\cdot)\left[\left(\widetilde{C}(\cdot)+\widetilde{F}(\cdot)\right) \bar{x}^*(\cdot)+\widetilde{D}(\cdot) u^{*(N)}(\cdot)+\widetilde{g}(\cdot)\right]+\zeta(\cdot).
	\end{aligned}
	$$
	Applying It\^o's formula to $\left\langle \widetilde{x}_i(\cdot), \lambda_i(\cdot)\right\rangle+\frac{1}{N}\sum_{j=1}^N\left\langle \widetilde{x}_j(\cdot), \lambda^i_j(\cdot)\right\rangle$, we derive
	\begin{equation}\label{Ito of I_i}
		\begin{aligned}
			& \left\langle G \widetilde{x}_i(T)-G \Gamma_0 \widetilde{x}^{(N)}(T), \bar{x}_i^*(T)-\Gamma_0 \bar{x}^{*}(T)-\eta_0\right\rangle\\
			=&\ \mathbb{E} \int_0^T\bigg[-\left\langle Q\left(\widetilde{x}_i-\Gamma_1 \widetilde{x}^{(N)}\right), \bar{x}_i^*-\Gamma_1 \bar{x}^{*}-\eta_1\right\rangle\\
			&\qquad -\left\langle \widetilde{u}_i, \left[B^\top P+D^\top PC+\widetilde{D}^\top P\widetilde{C}-\left(D^\top PD+\widetilde{D}^\top P\widetilde{D}\right)\mathcal{R}^{-1}\widetilde{P}\right]\bar{x}_i^*\right.\\
			&\qquad +\left[B^\top K+D^\top PF+\widetilde{D}^\top P\widetilde{F}+\widetilde{D}^\top K\left(\widetilde{C}+\widetilde{F}\right)-\left(D^\top PD+\widetilde{D}^\top P\widetilde{D}\right)\mathcal{R}^{-1}\widetilde{P}\right.\\
			&\qquad \left.-\left(D^\top PD+\widetilde{D}^\top \left(P+K\right)\widetilde{D}\right)\left(\mathcal{R}+\widetilde{D}^\top K\widetilde{D}\right)^{-1}\left(\widetilde{P}+\widetilde{K}\right)\right]\bar{x}^*\\
			&\qquad \left.+\left[\Phi-\left(D^\top PD+\widetilde{D}^\top \left(P+K\right)\widetilde{D}\right)\left(\mathcal{R}+\widetilde{D}^\top K\widetilde{D}\right)^{-1}\left(\widetilde{P}+\widetilde{K}\right)\Phi\right]\right\rangle\\
			&\qquad +\frac{1}{N}\left\langle  \tilde{u}_i, B^\top\lambda^i_j+\widetilde{D}^\top\theta \right\rangle\bigg] dt.
		\end{aligned}
	\end{equation}
	Substituting (\ref{Ito of I_i}) into (\ref{I_i 1}), we obtain
	\begin{equation}\label{I_i 2}
		\begin{aligned}
			I_i&= \mathbb{E}\left[ \int_0^T\left(\left\langle Q\left(\widetilde{x}_i-\Gamma_1 \widetilde{x}^{(N)}\right), \hat{x}_i^*-\bar{x}_i^*-\Gamma_1 \left(\hat{x}^{*(N)}-\bar{x}^{*}\right)\right\rangle+\frac{1}{N}\left\langle  \tilde{u}_i, B^\top\lambda^i_j+\widetilde{D}^\top\theta \right\rangle\right) dt\right.\\
			&\qquad + \left\langle G\left(\widetilde{x}_i(T)-\Gamma_0 \widetilde{x}^{(N)}(T)\right), \hat{x}_i^*(T)-\bar{x}_i^*(T)-\Gamma_0 \left(\hat{x}^{*(N)}(T)-\bar{x}^{*}(T)\right)\right\rangle\bigg].
		\end{aligned}
	\end{equation}
	Then, using the estimates (\ref{CC2}) and (\ref{CC3}), we can derive that $I_i=O\left(\frac{1}{N}\right)$. And the proof is completed.
\end{proof}

\subsection{Comparison with the fixed-point method}

For the fixed-point method, we first assume that the state-average $x^{(N)}(\cdot)$ can be approximated by some $\mathcal{F}_t^0$-adapted processes $z(\cdot)$ that will be subsequently determined through a so-called {\it consistency condition} (CC) system. We introduce the following auxiliary state $\check{x}_i(\cdot)\in L^2_{\mathcal{F}_t^i}(0,T;\mathbb{R}^n)$ which satisfies the following linear SDE
\begin{equation}\label{state limiting}
	\left\{\begin{aligned}
		d \check{x}_i=&\left(A \check{x}_i+B u_i +E z+f\right) d t+  \left(C \check{x}_i +D u_i +F z+g\right) d W_i
		\\&+\left(\widetilde{C} \check{x}_i +\widetilde{D} u_i +\widetilde{F} z+\widetilde{g}\right) d W_0, \\
		\check{x}_i (0)=&\ \xi_i,
	\end{aligned}\right.
\end{equation}
and the auxiliary limiting cost functional
\begin{equation}\label{cost functional limiting}
	\begin{aligned}
		J_i\left(u_i(\cdot)\right)=&\ \frac{1}{2}\mathbb{E}\left[ \int_0^T\left(\left\|\check{x}_i(t)-\Gamma_1(t) z(t)-\eta_1(t)\right\|_Q^2+\left\|u_i(t)-\eta_2(t)\right\|_R^2\right) d t\right.\\
		&\qquad +\left\|\check{x}_i(T)-\Gamma_0 z(T)-\eta_{0}\right\|_{G}^2\bigg].
	\end{aligned}
\end{equation}
And the admissible control set is still defined as $\mathcal{U}^d_i$. And we solve the following auxiliary control problem.

\begin{myprob}\label{problem limiting}
	Finding a strategy $u^*_i(\cdot)$ for each agent $\mathcal{A}_i$ satisfying
	\begin{equation*}
		J_i\left(u^*_i(\cdot)\right)=\inf _{u_i(\cdot) \in\, \mathcal{U}_i^{d}} J_i\left(u_i(\cdot)\right).
	\end{equation*}
\end{myprob}

We employ a variation method similar to Theorem \ref{theorem open loop centralized} to provide the open-loop strategy for Problem 3.1. Therefore, we omit the proof.

\begin{mythm}\label{theorem open loop fixed point}
	Let Assumptions \ref{A1} hold, then Problem \ref{problem limiting} has an optimal strategy $ \check{u}^*_i(\cdot)$ if and only if for $J_i\left(u_i(\cdot)\right)$ is convex in $u_i$ and the following FBSDEs
	\begin{equation}\label{Hamiltonian system fixed point}
		\left\{\begin{aligned}
			d \check{x}^*_i=&\left(A \check{x}^*_i+B  \check{u}^*_i+E z+f\right) d t+  \left(C \check{x}^*_i+D  \check{u}^*_i+F z+g\right) d W_i \\
			&+\left(\widetilde{C}\check{x}^*_i+\widetilde{D}  \check{u}^*_i
			+\widetilde{F}z+\widetilde{g}\right) d W_0, \\
			d m_i =
			&-\bigg( A^\top m_i + C^\top n_i + \widetilde{C}^\top n_{i0}+\left(I_n-\frac{\Gamma_1}{N}\right)^\top Q \big( \check{x}_i^* - \Gamma_1 z- \eta_1 \big) \bigg) dt \\
			&   +  n_i\, dW_i + n_{i0} \, dW_0,\\
			\check{x}_i^*(0)&=\xi_i,\\
			m_i(T)&= G\left(\check{x}^*_i(T)-\Gamma_0 z(T)-\eta_0\right),
		\end{aligned}\right.
	\end{equation}
	admit adapted solution $(\check{x}_i^*(\cdot), m_i(\cdot), n_i(\cdot), n_{i0}(\cdot))\in L^2_{\mathcal{F}^i_t}(0,T;\mathbb{R}^n)\times L_{\mathcal{F}^i_t}^2\left(0, T ; \mathbb{R}^n\right) \times L_{\mathcal{F}^i_t}^2\left(0, T ; \mathbb{R}^{n}\right)\times L_{\mathcal{F}^i_t}^2\left(0, T ; \mathbb{R}^{n}\right)$, satisfying the stationary condition
	\begin{equation}\label{optimal open loop condition of fixed point}
		B^\top m_i+D^\top n_i+\widetilde{D}^\top n_{i0}+R \check{u}^*_i-R\eta_2=0.
	\end{equation}
\end{mythm}

Since the $\mathcal{F}_t^{i}$-adapted process $\check{x}_i^*(\cdot)$ and $\mathcal{F}_t^{j}$-adapted process $\check{x}_j^*(\cdot)$ (for $i\neq j$ and $i,j=1,2,\cdots$) are identically distributed and conditionally independent given $\mathbb{E}\left[\cdot \mid \mathcal{F}^{W_0}\right]$, we can apply the conditional strong law of large numbers (Majerek et al. \cite{Majerek-Nowak-Zieba-2005}) to draw a conclusion
\begin{equation}\label{SLLN}
	z(\cdot)=\lim _{N \rightarrow \infty} \frac{1}{N} \sum_{i=1}^N \check{x}_i^*(\cdot)=\mathbb{E}\left[\check{x}_i^*(\cdot) \mid \mathcal{F}_\cdot^{W_0}\right].
\end{equation}
Substituting $z(\cdot)$ with $\mathbb{E}\left(\check{x}^*_i(\cdot)\mid\mathcal{F}^{W_0}_\cdot\right)$, then we derive the following CC system, which is a {\it conditional mean field FBSDE} (CMF-FBSDE) (\cite{Shi-Wang-Xiong-2016}) of $(\check{x}_i^*(\cdot), m_i(\cdot), n_i(\cdot), n_{i0}(\cdot))\in L^2_{\mathcal{F}^i_t}(0,T;\mathbb{R}^n)\times L_{\mathcal{F}^i_t}^2\left(0, T ; \mathbb{R}^n\right) \times L_{\mathcal{F}^i_t}^2\left(0, T ; \mathbb{R}^{n}\right)\times L_{\mathcal{F}^i_t}^2\left(0, T ; \mathbb{R}^{n}\right)$ :
\begin{equation}\label{Hamiltonian system fixed point CC}
	\left\{\begin{aligned}
		d \check{x}^*_i=&\left(A \check{x}^*_i+B  \check{u}^*_i+E \mathbb{E}\left[\check{x}_i^* \mid \mathcal{F}_t^{W_0}\right]+f\right) d t+  \left(C \check{x}^*_i+D  \check{u}^*_i+F \mathbb{E}\left[\check{x}_i^* \mid \mathcal{F}_t^{W_0}\right]+g\right) d W_i \\
		&+\left(\widetilde{C}\check{x}^*_i+\widetilde{D}  \check{u}^*_i
		+\widetilde{F}\mathbb{E}\left[\check{x}_i^* \mid \mathcal{F}_t^{W_0}\right]+\widetilde{g}\right) d W_0, \\
		d m_i =
		&-\bigg(A^\top m_i+C^\top n_i+\widetilde{C}^\top n_{i0}+\left(I_n-\frac{\Gamma_1}{N}\right)^\top Q \big( \check{x}_i^* - \Gamma_1 \mathbb{E}\left[\check{x}_i^* \mid \mathcal{F}_t^{W_0}\right]- \eta_1 \big) \bigg) dt \\
		& + n_i\, dW_i + n_{i0} \, dW_0,\\
		\check{x}_i^*(0)&=\xi_i,\\
		m_i(T)&= G\left(\check{x}^*_i(T)-\Gamma_0 \mathbb{E}\left[\check{x}_i^* \mid \mathcal{F}_t^{W_0}\right](T)-\eta_0\right).\\
	\end{aligned}\right.
\end{equation}
Next, we employ a decoupling technique to derive the state feedback form of the optimal strategy of Problem \ref{problem limiting}.
For each $i=1, \cdots, N$, we suppose
\begin{equation}\label{decouple form fixed point}
	m_i(\cdot)=\check{P}(\cdot) \check{x}_i^*(\cdot)+\check{K}(\cdot) \mathbb{E}\left[\check{x}_i^*(\cdot) \mid \mathcal{F}_\cdot^{W_0}\right](\cdot)+\check{\varphi}(\cdot),
\end{equation}
where $\check{P}(\cdot),\check{K}(\cdot)$ are deterministic differentiable functions satisfying $\check{P}(T)= G$, $\check{K}(T)= -G\Gamma_0$ and $(\check{\varphi}(\cdot),\check{\beta}(\cdot))$ is an $\mathcal{F}^0_t$-adapted processes pair satisfying
\begin{equation}
	d \check{\varphi} =\check{\alpha} d t+\check{\beta} d W_0, \quad \check{\varphi}(T)=- G \eta_0.
\end{equation}
By applying It\^o's formula to (\ref{decouple form fixed point}) and comparing the diffusion coefficients, we get
\begin{equation}\label{diffusion coefficient of n_i}
	\left\{\begin{aligned}
		&\check{P}\left(C \check{x}^*_i+D  \check{u}^*_i+F \mathbb{E}\left[\check{x}_i^* \mid \mathcal{F}_t^{W_0}\right]+g\right)=n_i ,\\
		&\check{P}\left(\widetilde{C} x^*_i+\widetilde{D}  \check{u}^*_i+\widetilde{F} \mathbb{E}\left[\check{x}_i^* \mid \mathcal{F}_t^{W_0}\right]+\widetilde{g}\right)+\check{K}\left(\left(\widetilde{C}+\widetilde{F}\right) \mathbb{E}\left[\check{x}_i^* \mid \mathcal{F}_t^{W_0}\right]\right.\\
		&\left.+\widetilde{D} \mathbb{E}\left[u_i^* \mid \mathcal{F}_t^{W_0}\right]+\widetilde{g}\right)+\check{\beta} =n_{i0}.
	\end{aligned}\right.
\end{equation}
To obtain the feedback representation of the centralized open-loop Nash equilibrium, we need the following assumption.
\begin{myassump}\label{A6}
	$$
	\begin{aligned}
		& R+D^\top\check{P} D+\widetilde{D}^\top\check{P} \widetilde{D}>0,\\
		& R+D^\top\check{P} D+\widetilde{D}^\top\left(\check{P}+\check{K}\right)\widetilde{D}>0.
	\end{aligned}
	$$
\end{myassump}
Substituting (\ref{diffusion coefficient of n_i}) into (\ref{optimal open loop condition of fixed point}), we first get
$$
\mathbb{E}\left[\check{u}_i^* \mid \mathcal{F}_t^{W_0}\right]=-\left(\check{\mathcal{R}}+\widetilde{D}^\top \check{K} \widetilde{D}\right)^{-1}\left[\left(\check{\widetilde{P}}+\check{\widetilde{K}}\right) \mathbb{E}\left[\check{x}_i^* \mid \mathcal{F}_t^{W_0}\right]+ \check{\Phi}\right],
$$
where
$$
\begin{aligned}
	& \check{\mathcal{R}}:=R+D^\top\check{P} D+\widetilde{D}^\top \check{P} \widetilde{D}, \\
	& \check{\widetilde{P}}:=B^\top \check{P}+D^\top\check{P} C+\widetilde{D}^\top \check{P} \widetilde{C} ,\\
	& \check{\widetilde{K}}:=B^\top \check{K}+D^\top\check{P} F+\widetilde{D}^\top \check{P} \widetilde{F}+\widetilde{D}^\top \check{K}\left(\widetilde{C}+\widetilde{F}\right),\\
	& \check{\Phi}:=B^\top \check{\varphi}+D^\top\check{P} g+\widetilde{D}^\top\left(\check{P}+\check{K}\right) \widetilde{g}+\widetilde{D}^\top \check{\beta}-R\eta_2 ,
\end{aligned}
$$
and then the optimal strategy of Problem \ref{problem limiting} has the feedback representation
\begin{equation}\label{optimal feedback of open loop condition of fixed point}
	\check{u}_i^*=-\check{\mathcal{R}}^{-1} \check{\widetilde{P}} \check{x}_i^*+\left[\check{\mathcal{R}}^{-1} \check{\widetilde{P}}-\left(\check{\mathcal{R}}+\widetilde{D}^\top \check{K} \widetilde{D}\right)^{-1}\left(\check{\widetilde{P}}+\check{\widetilde{K}}\right)\right] \mathbb{E}\left[\check{x}_i^* \mid \mathcal{F}_t^{W_0}\right]-\left(\check{\mathcal{R}}+\widetilde{D} \check{K} \widetilde{D}\right)^{-1} \check{\Phi}.
\end{equation}
Then, by applying It\^o's formula to (\ref{decouple form fixed point}) and comparing the drift coefficients, we obtain the equation satisfied by the coefficients of the term $\check{x}^*_i(\cdot)$:
\begin{equation}\label{check P equation}
	\left\{\begin{aligned}
		&\dot{\check{P}}+\check{P} A+A^\top \check{P}+C^\top \check{P} C+\widetilde{C}^\top \check{P} \widetilde{C}-\check{\widetilde{P}}^\top \check{\mathcal{R}}^{-1} \check{\widetilde{P}}+Q=0 , \\
		&\check{P}(T)=G ,
	\end{aligned}\right.
\end{equation}
the equation satisfied by the coefficients of the term $\mathbb{E}\left[\check{x}_i^* \mid \mathcal{F}_t^{W_0}\right]$:
\begin{equation}\label{check K equation}
	\left\{\begin{aligned}
		&\dot{\check{K}}+\check{K}(A+E)+\check{P} E+A^\top \check{K}+C^\top \check{P} F+\widetilde{C}^\top\left(\check{P}+\check{K}\right) \widetilde{F} \\
		&+\check{\widetilde{P}}^\top \check{\mathcal{R}}^{-1} \check{\widetilde{P}}-\left(\check{\widetilde{P}}^\top+\check{K} B+\widetilde{C}^\top \check{K} \widetilde{D}\right)\left(\check{\mathcal{R}}+\widetilde{D}^\top \check{K} \widetilde{D}\right)^{-1}\left(\check{\widetilde{P}}+\check{\widetilde{K}}\right)-Q \Gamma_1=0,\\
		&\check{K}(T)=-G \Gamma_0,
	\end{aligned}\right.
\end{equation}
and the equation satisfied by the non-homogeneous term:
\begin{equation}\label{check varphi equation}
	\left\{\begin{aligned}
		d \check{\varphi}= & -\bigg\{A^\top \check{\varphi}+\left(\check{P}+\check{K}\right) f-\left[\left(\check{P}+\check{K}\right) B-C^\top \check{P} D-\widetilde{C}\left(\check{P}+\check{K}\right) \widetilde{D}\right] \\
		&\quad \left.\times\left(\check{\mathcal{R}}+\widetilde{D}^\top \check{K} \widetilde{D}\right)^{-1}\check{\Phi}+C^\top \check{P} g+\widetilde{C}^\top\left(\check{P}+\check{K}\right) \widetilde{g}+\widetilde{C}\check{\beta} -Q\eta_1\right\} d t+\check{\beta} d W_0 ,\\
		\check{\varphi}(T)= & -G \eta_0.
	\end{aligned}\right.
\end{equation}
Using a method similar to that in Li et al. \cite{Li-Nie-Wu-2023}, we can obtain the asymptotic optimality of strategy (\ref{optimal feedback of open loop condition of decentralized}). We omit the proof.
\begin{myremark}
	We find that equations (\ref{P equation})-(\ref{varphi equation}) and equation (\ref{check P equation})-(\ref{check varphi equation}) have the same form. Using the results of Theorem 3.3, we have $P(\cdot)=\check{P}(\cdot)$, $K(\cdot)=\check{K}(\cdot)$, $\varphi(\cdot)=\check{\varphi}(\cdot)$. Therefore, the decentralized strategy (\ref{optimal feedback of open loop condition of decentralized}) obtained using the direct method is identical to the decentralized strategy (\ref{optimal feedback of open loop condition of fixed point}) obtained using the fixed-point method.
\end{myremark}

\section{An application to production planning problem}

To illustrate the applications of studying stochastic LQ large-population problems with common noise, motivated by the example on page 55 of Yong and Zhou \cite{Yong-Zhou-1999}, we present the following production planning problem.

We consider that there are $N$ homogeneous manufacturers in the market. Each manufacturer produces a kind of product $\mathcal{B}_i$ over the interval $[0,T]$, so as to hold exactly $k$ units of inventory at the terminal time $T$. Then, every investor's objective is to choose a production planning strategy that minimizes the cost. We denote the inventory process of the $i$-th manufacturer by $I_i(\cdot)$. Then we assume that the loss $L_i(\cdot)$ of inventory of the $i$-th manufacturer can be approximated by the following stochastic differential equation:
\begin{equation*}
	dL_i(t)=-r(t) I_i(t) +\left(\tilde{a}(t) I_i(t)+g(t)\right) dW_0(t).
\end{equation*}
Here $r(\cdot) > 0$ is a deterministic function denoting the instantaneous loss rate of inventory (e.g., losses caused by environmental factors) of $i$-th manufacturer and the volatility $\left(\tilde{a}(\cdot) I_i(\cdot)+g(\cdot)\right)$ of losses of $i$-th manufacturer depends on the his inventory level. The common Brownian motion $W_0(\cdot)$ represents that all manufacturers' inventories are affected by the same external environment. We further assume that every manufacturer faces a demand of $d(\cdot)$ units. Therefore, we can assume the inventory process $I_i(\cdot)$ of $i$-th manufacturer satisfies the following stochastic differential equation: 
	\begin{equation}\label{inventory}
		\left\{\begin{aligned}
			dI_i(t) =& \left(-r(t)I_i(t)+c(t)u_i(t)+m(t)(I^{(N)}(t)-I_i(t))-d(t)\right)dt\\
			&+\left(\tilde{b}(t)u_i(t)+\tilde{g}(t)\right)dW_i(t)+\left(\tilde{a}(t)I_i(t)+g(t)\right)dW_0(t),\\
			I_i(0)=&I_0
		\end{aligned}\right.
	\end{equation}
	where $u_i(\cdot)$ is control process of $i$-th manufacturer which can regulate the instantaneous productivity, $c(\cdot)$ is the productivity adjustment coefficient, and $\left(\tilde{b}(\cdot) u_i(\cdot)+\tilde{g}\right)$ denotes the volatility of instantaneous productivity of $i$-th manufacturer. Here, $W_i$ is the private noise of the i-th manufacturer representing internal influence within the manufacturer. The $N$ manufacturers maintain mutual contact. They can mutually supply or borrow inventory from one another so that each individual inventory level tracks the common average inventory. In other word, if $I_i(\cdot)$ is less than $I^{(N)}(\cdot)$, i-th manufacturer will borrow an amount $(I^{(N)}(\cdot)-I_i(\cdot))$ of inventory from other manufacturers. Otherwise, i-th manufacturer will supply inventory instead.
	
	Then,  we consider the cost functional for the i-th manufacturer to be
	\begin{equation}
		\mathcal{J}_i(u_i(\cdot);u_{-i}(\cdot))=\mathbb{E}\left\{\int_{0}^{T}\left[ Q(t) \left(I_i(t)-I^{(N)}(t)\right)^2+R(t)\left( u_i(t)-\eta_2(t)\right)^2 dt\right]+G\left(I_i(T)-k\right)^2\right\}
	\end{equation}
	The running cost comprises two parts: the first term penalizes the deviation between the i-th manufacturer's inventory and the average inventory; the second term represents the deviation of the control from the reference productivity benchmark $\eta_2(\cdot)$. The terminal cost penalizes the deviation between the i-th manufacturer's inventory and the target inventory $k$ at the terminal time $T$.  Owing to real-world constraints and the curse of dimensionality, centralized strategies that require complete information are seldom attainable. We therefore draw on the model of Section 2 to formulate the following problem.
	\begin{myprob}
		To find a optimal production planning decentralized strategy
		set $u^*(\cdot)=\\ \left(u_1^*(\cdot),u_2^*(\cdot),\cdots,u_N^*(\cdot)\right)$, $u_i^*(\cdot)\in\mathcal{U}^d_i, i=1,\cdots, N$ such that 
		\begin{equation*}
			\mathcal{J}_i\left(u^*_i(\cdot), u^*_{-i}(\cdot\right)\leq\inf _{u_i(\cdot) \in\, \mathcal{U}_i^{c}} \mathcal{J}_i\left(u_i(\cdot), u^*_{-i}(\cdot)\right)+\varepsilon,
		\end{equation*}
		where $\varepsilon=O(\frac{1}{\sqrt{N}})$.
	\end{myprob}
	Obviously, the production planning model established above is a special case of the model introduced in Section 2. Therefore,  we apply the main results from Section 3 to solve this production planning problem. For the simplicity of the calculations in this example, we set the number of manufacturers $N=300$ and time horizon $[0,1]$. We additionally assume the following system parameters: $r=0.1$, $c=0.5$, $m=0.3$, $d=2$, $\tilde{b}=0.1$, $g=0.5$,  $\tilde{a}=0.1$, $\tilde{g}=0.5$, $I_i$, $i=1,\cdots,N$ satisfies uniformly distributed on the interval $[2.5,3.5]$, $Q=1$, $R=10$, $\eta_2=6$, $G=1$, $k=2.5$. To provide more intuitive insight, we use the specific coefficients above to generate several plots that corroborate the validity of our practical findings. By the  Euler's method, we plot the solution curves of Riccati equations \ref{P equation}, \ref{K equation}, \ref{Pi equation} and \ref{M equation}
	in Figure \ref{fig:Riccati}. And then, the optimal production planning decentralized strategy of the N manufacturers are given by Figure \ref{fig:bar{u}^*_i}. After implementing the optimal decentralized strategy, the optimal states of N manufacturers are shown in Figure \ref{fig:bar{x}^*_i}. Figure \ref{fig:hat{x}^{*(N)}} illustrates the convergence of the limiting process $\bar{x}^*$ to the state-average term $\hat{x}^{*(N)}$. 
	
	\begin{figure}
		\begin{center}
			\includegraphics[width=0.6\textwidth]{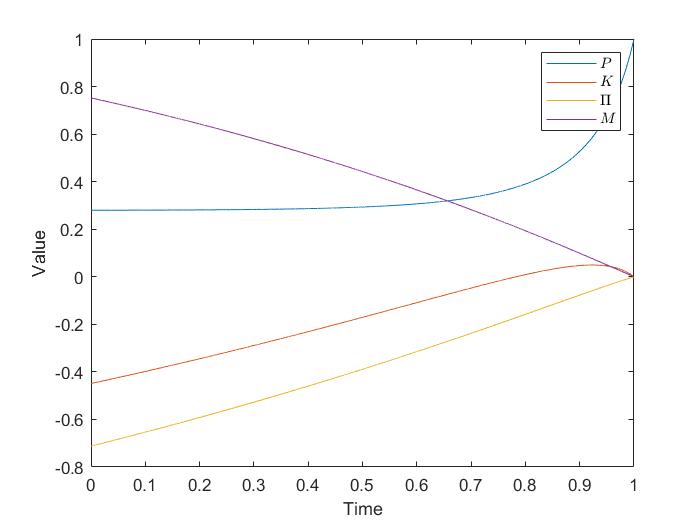}
			\caption{The solution curve of $P,K,\Pi,M$}
			\label{fig:Riccati}
		\end{center}
	\end{figure}
	\begin{figure}
		\begin{center}
			\includegraphics[width=0.6\textwidth]{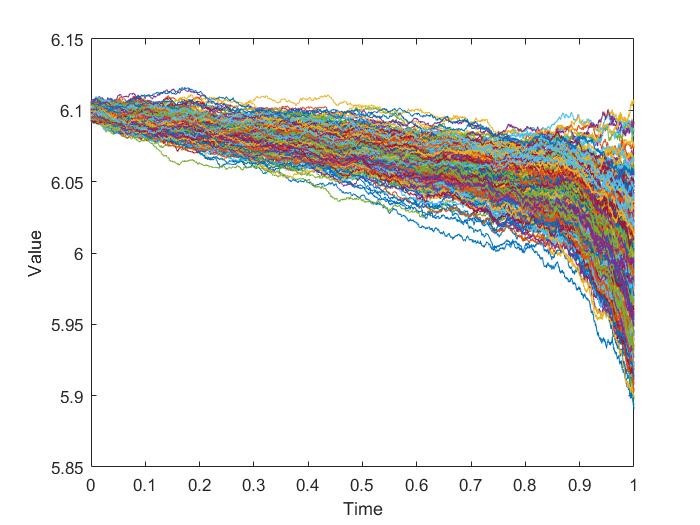}
			\caption{The solution curve of $\bar{u}^*_i$}
			\label{fig:bar{u}^*_i}
		\end{center}
	\end{figure}
	\begin{figure}
		\begin{center}
			\includegraphics[width=0.6\textwidth]{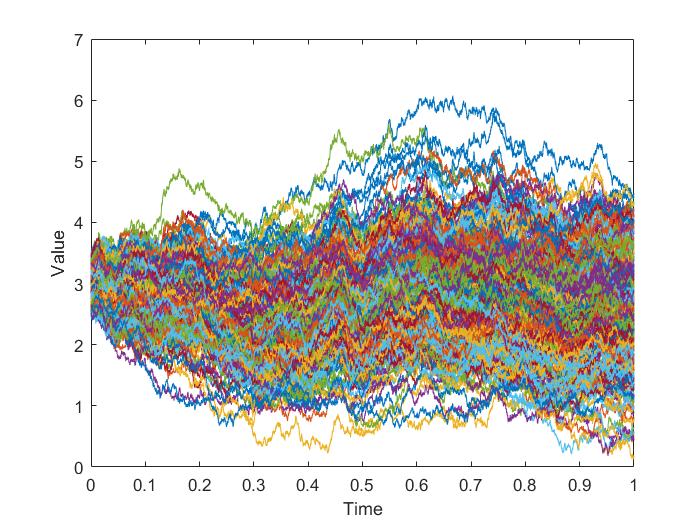}
			\caption{The solution curve of $\bar{x}^*_i$}
			\label{fig:bar{x}^*_i}
		\end{center}
	\end{figure}
	\begin{figure}
		\begin{center}
			\includegraphics[width=0.6\textwidth]{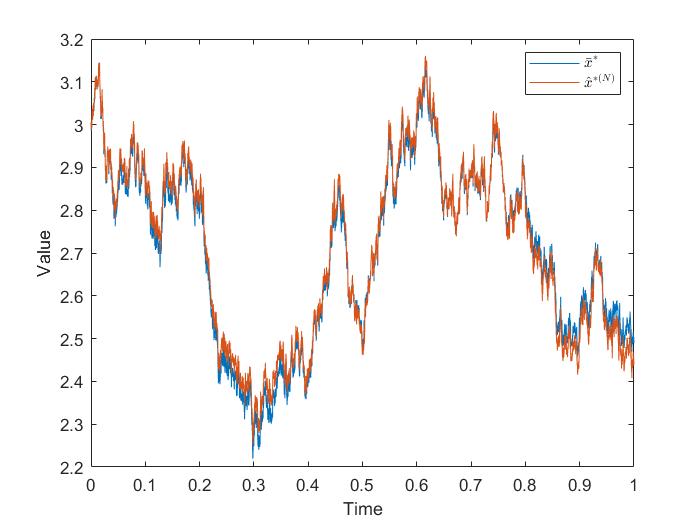}
			\caption{The solution curve of $\hat{x}^{*(N)}$ and $\bar{x}^{*}_i$}
			\label{fig:hat{x}^{*(N)}}
		\end{center}
	\end{figure}
	
\section{Conclusion}

Motivated some product planning problems in the market, in this paper we have applied a direct method to consider a class of general LQ MFGs with common noise. By comparing with the fixed-point approach, we find that the decentralized open-loop asymptotic Nash equilibrium strategies obtained are identical to those derived using the direct method. Our present work suggests various future research directions. For example, (i) to study the mean-variance analysis with relative performance in our setting; (ii) to study the MFGs with constraint, we can attempt to adopt the method of Lagrange multipliers and the Ekeland variational principle; (iii) to consider the direct method to solve MFGs with the state equation contains control average terms. We plan to study these issues in our future works.

\end{document}